\documentclass[leqno, 12pt]{amsart}
\usepackage{cases}
\usepackage{amsmath}
 \usepackage{graphicx}
 \usepackage{mathrsfs}
\usepackage{bbm}
\usepackage{amsfonts}
\usepackage{amssymb}
\usepackage{color}
\usepackage{multirow}
\usepackage{enumerate}
\usepackage{epstopdf}
\usepackage{cite,stmaryrd,txfonts}
\usepackage{tikz}
\usepackage{slashbox}
\usepackage{makecell}
\usepackage{leftidx}
\usepackage{algorithmic}

\usetikzlibrary{arrows.meta}
\usetikzlibrary{patterns}

\usepackage{hyperref}
\usepackage[ruled,vlined]{algorithm2e}

\def\opn#1#2{\def#1{\operatorname{#2}}} 
\opn\chara{char} \opn\length{\ell}
\opn\projdim{proj\,dim} \opn\injdim{inj\,dim} \opn\rank{rank}
\opn\depth{depth} \opn\grade{grade} \opn\height{height}
\opn\embdim{emb\,dim} \opn\codim{codim}

\opn\Tr{Tr} \opn\bigrank{big\,rank}
\opn\superheight{superheight}\opn\lcm{lcm}
\opn\trdeg{tr\,deg}%
\opn\reg{reg} \opn\lreg{lreg}
\opn\Ker{Ker} \opn\Coker{Coker} \opn\Im{Im} \opn\Hom{Hom}
\opn\Tor{Tor} \opn\Ext{Ext} \opn\End{End} \opn\Aut{Aut} \opn\id{id}

\opn\nat{nat}
\opn\pff{pf}
\opn\Pf{Pf} \opn\GL{GL} \opn\SL{SL} \opn\mod{mod} \opn\ord{ord}

\def\Implies{\ifmmode\Longrightarrow \else
     \unskip${}\Longrightarrow{}$\ignorespaces\fi}
\def\implies{\ifmmode\Rightarrow \else
     \unskip${}\Rightarrow{}$\ignorespaces\fi}
\def\iff{\ifmmode\Longleftrightarrow \else
     \unskip${}\Longleftrightarrow{}$\ignorespaces\fi}

\let\:=\colon

\DeclareMathOperator*{\argmax}{argmax}
\DeclareMathOperator*{\argmin}{argmin}

\newtheorem{Theorem}{Theorem}[section]

\newtheorem{Remark}[Theorem]{Remark}

\newtheorem{Example}[Theorem]{Example}

\newtheorem{Algorithm}[Theorem]{Algorithm}

\theoremstyle{definition}

\opn\ini{in} \opn\inm{inm} \opn\Sym{Sym} \opn\diag{diag}
\opn\Ii{(i)} \opn\Iii{(ii)}

\title{ A new finite element method for elliptic optimal control problems with pointwise state constraints in energy spaces}
\author{Wei Gong$^\dag$ and Zhiyu Tan$^\diamond$}
\thanks{ $^\dag$NCMIS \& LSEC, Institute of Computational Mathematics, Academy of Mathematics and Systems Science, Chinese Academy of Sciences, China. Email: {\tt wgong@lsec.cc.ac.cn}. The author was supported in part by the Strategic Priority Research Program of the Chinese Academy of Sciences (grant XDB 41000000) and the NSFC (grant 12071468).}
\thanks{$^\diamond$Center for Computation and Technology, Louisiana State University, Baton Rouge, LA 70803, USA.
	Email: {\tt ztan@cct.lsu.edu} or {\tt zhiyutan@amss.ac.cn}}
	
\begin{document}
\maketitle
{\bf Abstract:}\hspace*{10pt} {In this paper we propose a new finite element method for solving elliptic  optimal control problems with pointwise state constraints, including the distributed controls and the Dirichlet or Neumann boundary controls. The main idea is to use energy space regularizations in the objective functional, while the equivalent representations of the energy space norms, i.e., the $H^{-1}(\Omega)$-norm for the distributed control, the $H^{1/2}(\Gamma)$-norm for the Dirichlet control and the $H^{-1/2}(\Gamma)$-norm for the Neumann control, enable us to transform the optimal control problem into an elliptic variational inequality involving only the state variable. The elliptic variational inequalities are second order for the three cases, and include additional equality constraints for Dirichlet or Neumann boundary control problems. Standard $C^0$ finite elements can be used to solve the resulted variational inequality. We provide preliminary a priori error estimates for the new algorithm for solving distributed control problems. Extensive numerical experiments are carried out to validate the accuracy of the new algorithm.  }

{{\bf Keywords:}\hspace*{10pt} optimal control problem, distributed control, boundary control, pointwise state constraints, energy space method, variational inequality, finite element }

{\bf Subject Classification:} 49J20, 65K10, 65N12, 65N15, 65N30.

\section{Introduction}
\setcounter{equation}{0}

Let $\Omega\subset \mathbb{R}^d$ $(d=2,3)$ be a bounded domain with Lipschitz boundary $\Gamma$. Denote by $Y\subset H\subset Y^*$ a Gelfand triple of Hilbert spaces $Y$ and $H$ defined in $\Omega$. In this paper we consider the following optimal control problem with pointwise state constraints
\begin{eqnarray}
\min\limits_{u\in U,\ y\in K\subset Y}\ J(y,u)={1\over 2}\|y-y_d\|_{H}^2+{\alpha\over 2}\|u\|_{U}^2\quad\mbox{subject\ to}\quad y=Su,\label{abstract_OCP}
\end{eqnarray}
where $y$ is the state variable and $u$ is the control variable, $S:U\rightarrow Y$ is the control to state mapping which will be given as the solution operator of a second order elliptic equation. $y_d\in H$ is the target state and $\alpha>0$ is the regularization parameter. $K\subset Y$ is a closed convex subset representing the state constraints. $Y$ is referred to as the state space and $U$ is the control space. 

Assume that $S:U\rightarrow Y$ is an isomorphism, then we have the norm equivalence
\begin{eqnarray}
\|u\|_U\approx \|Su\|_Y\quad\forall u\in U,\nonumber
\end{eqnarray}
where $\|\cdot\|_Y$ denotes the norm of $Y$. Then we are led to the equivalent optimization problem involving only the state $y$:
\begin{eqnarray}
\min\limits_{y\in K\subset Y}\ J(y)={1\over 2}\|y-y_d\|_{H}^2+{\alpha\over 2}\|Su\|_{Y}^2.\label{OCP_equiv}
\end{eqnarray}
In the following we will specify the precise settings for the distributed control and the Dirichlet/Neumann boundary control. 

PDE-constrained optimal control problem is a hot research topic in the past few decades, due to its extensive applications in economy, fluid dynamics, chemical reaction process, shape and topology optimization, to name just a few, and its close connections to PDEs, optimization, control theory and other related disciplines. As is well known, numerical methods play the key role to put these applications into practice. During the last decades, great efforts and achievements have been made for the efficient solving of  PDE-constrained optimal control problems. We refer to \cite{HinzePinnauUlbrich} and the references therein for recent advances. 

Compared to control constrained optimal control problems, PDE-constrained optimal control with pointwise state constraints is much more difficult to analyze from both the theoretical and numerical points of view. The main difficulty comes from the subtle existence results of Lagrange multipliers and their extremely lower regularity. Under the classical formulations which rely on the first order optimality condition, the Lagrange multipliers associated with the pointwise state constraints are only measures (cf. \cite{Casas1997,Casas}), which yields non-smooth solutions and introduces difficulty in theoretical and numerical analysis of the first order optimality system (cf. \cite{HinzePinnauUlbrich}). The situation becomes even worse for boundary control problems because they usually have less smooth solutions compared to distributed control problems, which may yield restrictions on the considered spatial dimensions and locations where to impose state constraints (\cite{Casas1997,KrumbiegelMeyerRosch,MateosNeitzel}). However, boundary controls have extensive applications in PDE-constrained control problems because they are easy to implement compared to distributed ones, especially in fluid control problems (cf. \cite{GongMateosSinglerZhang}). 

The numerical solutions to boundary control problems with or without control constraints have been extensively studied, we refer to \cite{ApelMateosPfeffererRosch,CasasRaymond,ChowdhuryGudiNandakumaran,DeckelnickGuntherHinze,GongLiuTanYan,GongMateosSinglerZhang,GongYan,OfPhanSteinbach,Winkler} for Dirichlet boundary control problems and to \cite{ApelPfeffererRosch,ApelSteinbachWinkler,CasasMateos,CasasMateosTroltzsch,HinzeMatthes} for Neumann boundary control problems. Although most of the results chosen the control space $L^2(\Gamma)$ for both Dirichlet and Neumann boundary control problems, energy space methods were also used, i.e., $H^{1/2}(\Gamma)$-norm for the Dirichlet case (cf. \cite{OfPhanSteinbach,GongLiuTanYan,ChowdhuryGudiNandakumaran,GongMateosSinglerZhang,Winkler}) and $H^{-1/2}(\Gamma)$-norm for the Neumann case (cf. \cite{ApelSteinbachWinkler}). However, when considering pointwise state constraints, the numerical analysis for  boundary control problems are quite rare. This is in contrast to the fruitful results for distributed control problems, see, e.g., \cite{BrennerGedickeSung,BrennerSung,DeckelnickGuntherHinze,DeckelnickHinze,HintermullerHinze,LiuGongYan,Meyer}. To the best of our knowledge, we are only aware of \cite{JohnWachsmuth,MateosNeitzel} for Dirichlet boundary control problems and \cite{KrumbiegelRosch,KrumbiegelMeyerRosch} for Neumann boundary control problems where the pointwise state constraints were imposed in an interior subdomain. 

In principle, for linear-quadratic optimal control problems with pointwise state constraints the resulting discrete optimization problem can be solved by the standard quadratic programming algorithm, as was done in \cite{DeckelnickHinze}. However, this is only possible for small scale problems because the state constraints were treated as control constraints implicitly. Another approach relies on penalizations, including the Moreau-Yosida regularization (cf. \cite{HintermullerHinze,HintermullerKunisch}) and Lavrentiev regularization (\cite{CherednickenkoKrumbiegelRosch,KrumbiegelRosch,MeyerPrufertTroltzsch,MeyerRoschTroltzsch}), which, however, introduces additional regularization errors. The Moreau-Yosida regularization can be readily extended to boundary control problems, while the Lavrentiev regularization is not so convenient. To overcome this difficulty, a virtual control was introduced in \cite{KrumbiegelRosch} and a source representation was used in \cite{TroltzschYousept}. The third approach is to transform the optimal control problem with pointwise state constraints to a fourth order elliptic variational inequality involving only the state variable, which was first proposed in \cite{LiuGongYan}, and was further analyzed and generalized in \cite{BrennerGedickeSung,BrennerSung,BrennerSungZhang2019,BrennerSungZhang2015}, see also the excellent survey \cite{Brenner} and the references cited therein. The latter approach was also used to solve parabolic optimal control problems with pointwise state constraints in \cite{GongHinze}. However, this approach can not be used directly for boundary control problems. 

In this paper, we provide such an extension of the above third approach to Dirichlet or Neumann boundary control problems with pointwise state constraints. We also apply the similar approach to elliptic distributed control problems with pointwise state or gradient state constraints. In contrast  to the reformulated fourth order elliptic variational inequality proposed in \cite{LiuGongYan}, our new approach involves only second order variational inequalities for distributed control problems. Compared to other approaches, our new algorithm has the following advantages: we do not assume constraint qualification conditions, while the classical approaches usually assume the Slater condition so that the state constraints have to be understood in the continuous space. This poses restrictions on the dimension and the convexity of the domain, as well as the problem data. Our approach is more flexible and the optimal control problem can be set up in different situations, e.g., non-convex domains or more involved constraints,where the only requirement is the well-posedness of the second order elliptic equation in the standard $H^1(\Omega)$ space. 
 
The remaining of this paper is organized as follows. In Section 2 we reformulate the distributed optimal control problem with pointwise state or gradient state constraints to second order elliptic variational inequalities. In Section 3 we consider the reformulations of the Dirichlet or Neumann boundary control problems. Several algorithms for solving the associated variational inequality with equality constraints are also presented. In Section 4 we present the finite element approximation to the elliptic variational inequality derived in Section 3  and give a preliminary error estimate for the variational inequality resulted from the distributed control problem. We also discuss how to recover the discrete controls once we obtain the discrete state. In the last section we carry out extensive numerical experiments to validate the accuracy of the new method. 

We follow standard notation for differential operators, function spaces and norms that can be found for example in \cite{Adamas_2003,Brenner_2008,Ciarlet_1978}. We also denote $C$ (with or without subscript) a generic positive constant which may stand for different values at its different occurrences but does not depend on the mesh size.

\section{Reformulations of the PDE-constrained distributed control problems}
\setcounter{equation}{0}
In this section we consider PDE-constrained distributed control problems in energy spaces, with pointwise constraints on the state or state gradients. By using the equivalent representation of the energy norm of the control, we reformulate the control problem into a second order elliptic variational inequality involving only the state variable. 

To this end, we first introduce a bilinear form $a:H^1(\Omega)\times H^1(\Omega)\rightarrow \mathbb{R}$ associated with a second order elliptic operator $L$:
\begin{eqnarray}
a(y,v):=\langle Ly,v\rangle_{H^1(\Omega)^*,H^1(\Omega)}\quad \forall y,v\in H^1(\Omega).\nonumber
\end{eqnarray}
Then for simplicity we set  $a(y,v)=(\nabla y,\nabla v)$ for Dirichlet boundary value problems and $a(y,v)=(\nabla y,\nabla v)+(y,v)$ for Neumann boundary value problems.

\subsection{Distributed control problems with pointwise state constraints}
Note that the distributed control problem with pointwise state constraints in the control space $L^2(\Omega)$ can be reformulated as a  fourth order elliptic variational inequality, see, e.g., \cite{LiuGongYan}.  However, with the energy space $H^{-1}(\Omega)$ which is sufficient to define a weak solution in $H_0^1(\Omega)$, we are able to transform the optimal control problem into a second order elliptic variational inequality. This is more favorable than the fourth order variational inequality derived from the control space $L^2(\Omega)$.

In this aspect, we consider the distributed control problem in energy space (cf. \cite{LangerSteinbachYang,NeumullerSteinbach} for an unconstrained elliptic optimal control problem)
 \begin{eqnarray}
\min\limits_{u\in H^{-1}(\Omega)}\quad J(y,u)={1\over 2}\|y-y_d\|_{0,\Omega}^2+{\alpha\over 2}\|u\|_{H^{-1}(\Omega)}^2
\end{eqnarray}
subject to 
\begin{equation}\label{state_distributed}
\left\{
\begin{aligned}
Ly:=-\Delta y=f+u\quad&\mbox{in}\ \Omega,\\
y=0\quad&\mbox{on}\ \Gamma,\\
y\leq y_b\quad&\mbox{in}\ \Omega.
\end{aligned}
\right.
\end{equation}
We suppose that $y_b\in H^1(\Omega)$ and $y_b>0$ so that the state constraints are not active on $\Gamma$. 

Recall that for the $H^{-1}(\Omega)$ space we have the equivalent norm definition (cf. \cite{NeumullerSteinbach})
\begin{eqnarray}
\|u\|_{H^{-1}(\Omega)}^2:=a(y_u,y_u)=\|\nabla y_u\|_{0,\Omega}^2,
\end{eqnarray}
where $y_u\in H_0^1(\Omega)$ satisfies
\begin{equation}\label{state_dist_yu}
\left\{
\begin{aligned}
-\Delta y_u=u\quad&\mbox{in}\ \Omega,\\
y_u=0\quad&\mbox{on}\ \Gamma.
\end{aligned}
\right.
\end{equation}
Let $y_f\in H_0^1(\Omega)$ solve
\begin{equation}\label{state_Dirichlet_f}
\left\{
\begin{aligned}
-\Delta y_f=f\quad&\mbox{in}\ \Omega,\\
y_f=0\quad&\mbox{on}\ \Gamma.
\end{aligned}
\right.
\end{equation}
Then we have $y=y_u+y_f$ and we are led to an equivalent optimization problem
\begin{eqnarray}
\min\limits_{y_u\in K^0}\quad J(y_u)={1\over 2}\|y_u+y_f-y_d\|_{0,\Omega}^2+{\alpha\over 2}\|\nabla y_u\|_{0,\Omega}^2,
\end{eqnarray}
whose Euler-Lagrange equation can be characterized as: Find $y_u\in K^0$ such that
\begin{eqnarray}
\alpha a(y_u,v-y_u)+(y_u,v-y_u)\geq (y_d-y_f,v-y_u)\quad\forall v\in K^0,\label{VI_distributed}
\end{eqnarray}
where $K^0:=\{y\in H_0^1(\Omega):\ y\leq y_b-y_f\ \mbox{a.e.\ in}\ \Omega\}$. Then with the energy space method we can obtain a second order elliptic variational inequality. 

Now we are going to derive the associated differential equations satisfied by the state $y_u$. We follow the standard approach and sketch the proof for simplicity. More details can be found in \cite{GlowinskiLions,KinderlehrerStampacchia}. By using integration by parts we have
\begin{eqnarray}
\int_\Omega (-\alpha\Delta y_u+y_u+y_f-y_d)(v-y_u)dx\geq 0\quad\forall v\in H_0^1(\Omega),\ v\leq y_b-y_f.\label{distributed_IbP}
\end{eqnarray}
Let $v=y_u+\psi$ with $\psi\in\mathcal{D}(\Omega)$ and $\psi\leq 0$. It is clear that $v\in H_0^1(\Omega)$  and $v\leq y_b-y_f$.
Inserting $v$ into the above inequality \eqref{distributed_IbP} we have
\begin{eqnarray}
\int_\Omega (-\alpha\Delta y_u+y_u+y_f-y_d)\psi dx\geq 0\quad\forall \psi\in\mathcal{D}(\Omega),\ \psi\leq 0.\nonumber
\end{eqnarray}
Therefore,
\begin{eqnarray}
-\alpha\Delta y_u+y_u+y_f-y_d\leq 0\quad\mbox{a.e.\ in}\ \Omega.\label{distributed_less}
\end{eqnarray}

Now we define the active and inactive sets
\begin{eqnarray}
\Omega^-:=\{x\in \Omega: y_u(x)<(y_b-y_f)(x)\},
\quad\Omega^0:=\{x\in \Omega: y_u(x)=(y_b-y_f)(x)\}.\nonumber
\end{eqnarray}
Then for any $\theta \in \mathcal{D}(\Omega)$, $\theta(x)\in [0,1]$, we have $v=\theta (y_b-y_f)(x)+(1-\theta)y_u\in H_0^1(\Omega)$ and $v\leq y_b-y_f$. Inserting $v$ into \eqref{distributed_IbP} we have
\begin{eqnarray}
\int_\Omega (-\alpha\Delta y_u+y_u+y_f-y_d)\theta(y_b-y_f-y_u) dx\geq 0\quad\forall \theta\in\mathcal{D}(\Omega),\ 0\leq \theta\leq 1.\nonumber
\end{eqnarray}
That is 
\begin{eqnarray}
\int_{\Omega^-} (-\alpha\Delta y_u+y_u+y_f-y_d)\theta(y_b-y_f-y_u) dx\geq 0\quad\forall \theta\in\mathcal{D}(\Omega),\ 0\leq \theta\leq 1.\nonumber
\end{eqnarray}
Therefore, we conclude that
\begin{eqnarray}
-\alpha\Delta y_u+ y_u+y_f-y_d\geq 0\quad\mbox{a.e.\ in}\ \Omega^-.\label{distributed_great}
\end{eqnarray}
This combining with \eqref{distributed_less} yields
\begin{eqnarray}
-\alpha\Delta y_u+ y_u+y_f-y_d= 0\quad\mbox{a.e.\ in}\ \Omega^-.
\end{eqnarray}

To summarize, we can characterize the variational inequality as follows
\begin{eqnarray}
-\alpha\Delta y_u+ y_u+y_f=y_d\quad\mbox{a.e.\ in}\ \Omega^-,\quad
-\alpha\Delta y_u+y_u+y_f-y_d\leq 0\quad\mbox{a.e.\ in}\ \Omega^0.\nonumber
\end{eqnarray}
Then we can introduce the slack variable $\mu\in H^{-1}(\Omega)$ such that
\begin{eqnarray}
\alpha a(y_u,v)+(y_u,v)-\langle \mu,v\rangle = (y_d-y_f,v)\quad \forall v\in H_0^1(\Omega)
\end{eqnarray}
and the complementarity conditions hold: $\mu\leq 0$, $\langle \mu,y_u+y_f-y_b\rangle_{H^{-1},H^1}=0$. By using the well-known regularity theory for second order elliptic obstacle problems we can expect that $y_u\in H^2(\Omega)\cap H_0^1(\Omega)$ when $\Omega$ is a convex polygonal domain and $y_b\in H^2(\Omega)$. 

Compared to the standard approach which searches for controls in $L^2(\Omega)$, we remark that the energy space approach gives a Lagrange multiplier in $H^{-1}(\Omega)$.

\begin{Remark}\label{rem:dis_LM_regularity}
The existence of a Lagrange multiplier in $H^{-1}(\Omega)$ can also be guaranteed by the result of classical optimization theory, see e.g., Example 1.7 in \cite{Ito_2008}. The derivation above also indicates that if $y_u\in H^2(\Omega)\cap H_0^1(\Omega)$, we actually have $\mu\in L^2(\Omega)$.
\end{Remark}

\subsection{Distributed control problems with pointwise gradient state constraints}

We can also consider the distributed control problem with pointwise constraints on the gradient of the state 
\begin{eqnarray}
\min\limits_{u\in U}\quad J(y,u)={1\over 2}\|y-y_d\|_{0,\Omega}^2+R(u)
\end{eqnarray}
subject to 
\begin{equation}\label{state}
\left\{
\begin{aligned}
-\Delta y=u\quad&\mbox{in}\ \Omega,\\
y=0\quad&\mbox{on}\ \Gamma,\\
|\nabla y|\leq y_b\quad&\mbox{in}\ \Omega.
\end{aligned}
\right.
\end{equation}
Here $R(u)$ is a regularization term and $|\cdot |$ denotes the standard Euclidean norm for vectors. 
The above problem was studied in \cite{DeckelnickGuntherHinze2009} with the regularization term $R(u)={\alpha\over q}\|u\|^q_{L^q(\Omega)}$, where $q>d$ was chosen such that $W^{2,q}(\Omega)\hookrightarrow C^1(\bar\Omega)$ holds. Therefore, the gradient state constraints can be understood in the continuous space and the Slater condition can be assumed. However, this setting was no longer valid for non-convex polygonal domain $\Omega$. We refer to \cite{GuntherHinze,OrtnerWollner,Wollner} for more results on elliptic distributed control problems with pointwise gradient state constraints. In \cite{BrennerSungWollner,BrennerSungWollner2020} the authors considered the approach of transforming the optimal control problem into a fourth order elliptic variational inequality with gradient constraints and proposed a $C^1$-conforming finite element method. However, only the one dimensional case was considered and the results can not be extended to higher dimensions straightforwardly. 

With the similar energy space method, we can choose $R(u)={\alpha\over 2}\|u\|^2_{H^{-1}(\Omega)}$ and understand the gradient state constraints in $L^2(\Omega)$. In this case, we can obtain the equivalent second order variational inequality problem: Find $y\in H_0^1(\Omega)$ such that $|\nabla y|\leq y_b$ and 
\begin{eqnarray}
\alpha a(y,v-y)+(y,v-y)\geq (y_d,v-y)\quad\forall v\in H_0^1(\Omega),\ |\nabla v|\leq y_b.
\end{eqnarray}

Assume that $\Omega\subset \mathbb{R}^d$ is bounded and convex with a Lipschitz boundary $\Gamma$. If $y_d\in L^p(\Omega)$, $1<p<\infty$, then we have the following regularity result (cf. \cite{GlowinskiLions})
\begin{eqnarray}
\Delta y\in L^p(\Omega),\ \ \ y\in W^{2,p}(\Omega),\ \ 1<p<\infty.\label{regularity}
\end{eqnarray}
For the gradient constraints $|\nabla y|\leq y_b$ we can define the active set $\Omega^0$ and inactive set $\Omega^-$ as follows:
\begin{eqnarray}
\Omega^0:=\{x\in \Omega:\ \ |\nabla y(x)|=y_b\},\quad
\Omega^-:=\Omega\backslash\Omega^0=\{x\in \Omega:\ \ |\nabla y(x)|<y_b\}.\nonumber
\end{eqnarray}
One can also prove that
\begin{eqnarray}
-\alpha\Delta y+y=y_d\ \ \ \mbox{a.e.\ in}\ \ \Omega^-;\quad
|\nabla y(x)|=y_b\ \ \ \mbox{a.e.\ in}\ \ \Omega^0.\nonumber
\end{eqnarray}

\subsection{A unified approach for distributed control problems}

In fact, for distributed control problems we can put the above approach and the approach proposed in \cite{LiuGongYan} in a unified framework. Denote by $Y$ and $U$ two Hilbert spaces defined in $\Omega$. Associated with these spaces we define a linear and bounded operator $A:Y\rightarrow U$. We assume that $A$ is an isomorphism between $Y$ and $U$. Then we have the following equivalent representation of norms in $U$ and $Y$
\begin{eqnarray}
\|u\|_{U}^2\approx \|A^{-1}u\|_{Y}^2\quad \forall u\in U.\label{dual_norm}
\end{eqnarray}

We consider the abstract optimal control problem
\begin{eqnarray}
\min\limits_{u\in U,\ y\in K}\ J(y,u)={1\over 2}\|\mathcal{I}_{Y\rightarrow H}y-y_d\|_{H}^2+{\alpha\over 2}\|u\|_{U}^2\ \mbox{subject\ to}\ Ay=u,\label{abstract_OCP}
\end{eqnarray}
where $K\subset Y$, $y$ is the state variable and $u$ is the control variable, $y_d\in H$ is the target state and $\mathcal{I}_{Y\rightarrow H}$ denotes the canonical injection of $Y$ into $H$. $K\subset Y$ is a closed convex subset. 

By using \eqref{dual_norm} the optimal control problem (\ref{abstract_OCP}) can be rewritten as
\begin{eqnarray}
\min\limits_{u\in U,\ A^{-1}u\in K}\ J(u)={1\over 2}\|\mathcal{I}_{Y\rightarrow H}A^{-1}u-y_d\|_{H}^2+{\alpha\over 2}\|A^{-1}u\|_Y^2.
\end{eqnarray}
Setting $y=A^{-1}u$. By viewing $y$ as the optimization variable, we are led to 
\begin{eqnarray}
\min\limits_{y\in K\subset Y}\ J(y)={1\over 2}\|\mathcal{I}_{Y\rightarrow H}y-y_d\|_{H}^2+{\alpha\over 2}\|y\|_Y^2.\label{abstract_OPT}
\end{eqnarray}
The Euler-Lagrange equation of the above optimization problem reads: Find $y\in Y$ such that $y\in K$ and
\begin{eqnarray}
\alpha (y,v-y)_Y+(\mathcal{I}_{Y\rightarrow H}y,\mathcal{I}_{Y\rightarrow H}(v-y))_H\geq (y_d, \mathcal{I}_{Y\rightarrow H}(v-y))_H
\end{eqnarray}
holds for any $v\in K\subset Y$, where $(\cdot,\cdot)_Y$ denotes the inner product in $Y$. 

When we choose $U=H^{-1}(\Omega)$ and $Y=H_0^1(\Omega)$, we known that $A=(-\Delta)$ is an isomorphism between $Y$ and $U$. Therefore, $\|u\|_{H^{-1}(\Omega)}\approx \|y\|_{H_0^1(\Omega)}$, we are led to the  minimization problem
\begin{eqnarray}
\min\limits_{y\in K\subset Y}\ J(y)={1\over 2}\|\mathcal{I}_{Y\rightarrow H}y-y_d\|_{H}^2+{\alpha\over 2}\|\nabla y\|_{0,\Omega}^2.\nonumber
\end{eqnarray}
This recovers the approach proposed in the above subsections.

If we choose $U=L^2(\Omega)$ and $Y=\{y\in H_0^1(\Omega):\Delta y\in L^2(\Omega)\}$,  then $A=(-\Delta)$ is also an isomorphism between $Y$ and $U$. Since $Y$ is a Hilbert space with the norm $\|y\|_Y=\|y\|_{0,\Omega}+\|\Delta y\|_{0,\Omega}\approx \|\Delta y\|_{0,\Omega}$ by the Poincar\'e-Friedrichs inequality, we have $\|u\|_{0,\Omega}\approx \|y\|_{Y}\approx \|\Delta y\|_{0,\Omega}$. Therefore, we are led to the equivalent minimization problem 
\begin{eqnarray}
\min\limits_{y\in K\subset Y}\ J(y)={1\over 2}\|\mathcal{I}_{Y\rightarrow H}y-y_d\|_{H}^2+{\alpha\over 2}\|\Delta y\|_{0,\Omega}^2.\nonumber
\end{eqnarray}
This is exactly the approach proposed in \cite{LiuGongYan}. 

\section{Reformulations of the PDE-constrained boundary control problems}
\setcounter{equation}{0}
In this section we consider PDE-constrained boundary control problems in energy spaces with pointwise state constraints, including Dirichlet or Neumann boundary controls. By using the equivalent representation of the energy norm of the control, we reformulate the control problem into a second order elliptic variational inequality involving only the state variable, plus some additional equality constraint. 

\subsection{Dirichlet boundary control problems with pointwise state constraints}

In this subsection we consider the  Dirichlet boundary control problems with pointwise state constraints
\begin{eqnarray}
\min\limits_{u\in U}\quad J(y,u)={1\over 2}\|y-y_d\|_{0,\Omega}^2+{\alpha\over 2}\|u\|_U^2
\end{eqnarray}
subject to 
\begin{equation}\label{state_Dirichlet}
\left\{
\begin{aligned}
Ly:=-\Delta y=f\quad&\mbox{in}\ \Omega,\\
y=u\quad&\mbox{on}\ \Gamma,\\
y\leq y_b\quad&\mbox{in}\ \Omega.
\end{aligned}
\right.
\end{equation}
There are several choices for the control space $U$. The first one is that $U:=L^2(\Gamma)$ (cf. \cite{ApelMateosPfeffererRosch,CasasRaymond,DeckelnickGuntherHinze,GongYan} ). With this choice we have that the associated state variable $y:=Su\in H^{1/2}(\Omega)$ because $S:L^2(\Gamma)\rightarrow H^{1/2}(\Omega)$ is well-defined (cf. \cite{CasasRaymond}).  Therefore, the pointwise state constraints can only be understood in $L^2(\Omega)$. Instead, the authors in \cite{MateosNeitzel} imposed the pointwise state constraints in a subdomain $\Omega_0\subset\subset\Omega$ by using the higher interior regularity of elliptic equations.

The second choice is the so-called energy space, i.e., $U=H^{1/2}(\Gamma)$ (cf. \cite{ChowdhuryGudiNandakumaran,GongLiuTanYan,GongMateosSinglerZhang,OfPhanSteinbach,Winkler}). Then one can only expect that $y:=Su\in H^{1}(\Omega)$ because $S:H^{1/2}(\Gamma)\rightarrow H^{1}_\Delta(\Omega)=\{v\in H^1(\Omega),\ Lv=0\ \mbox{in}\ \Omega\}$ is well-defined. In this case we can not ensure that $y$ is continuous.

Now we consider the control space $U=H^{1/2}(\Gamma)$ and choose ${\alpha\over 2}|u|^2_{H^{1/2}(\Gamma)}$ in the objective functional. Recall that we have the equivalent representation of the $H^{1/2}(\Gamma)$ semi-norm (cf. \cite{OfPhanSteinbach}):
\begin{eqnarray}
|u|_{H^{1/2}(\Gamma)}^2:=\langle u,\mathcal{D}u\rangle_{H^{1/2}(\Gamma);H^{-1/2}(\Gamma)}=(\nabla y_u,\nabla y_u)=\|\nabla y_u\|^2_{0,\Omega},
\end{eqnarray}
where $\mathcal{D}u:=\gamma_1y_u$ is the Dirichlet-to-Neumann operator with $\gamma_1:H_\Delta^1(\Omega)\rightarrow H^{-1/2}(\Gamma)$ and  $y_u$ is the solution of the following equation with homogeneous right-hand side
\begin{equation}\label{state_Dirichlet_u}
\left\{
\begin{aligned}
-\Delta y_u=0\quad&\mbox{in}\ \Omega,\\
y_u=u\quad&\mbox{on}\ \Gamma.
\end{aligned}
\right.
\end{equation}
In fact, an integration by parts yields the result.  Let $y_f\in H_0^1(\Omega)$ be the solution of (\ref{state_Dirichlet_f}). Then we have $y=y_f+y_u$. Note that $S:H^{1/2}(\Gamma)\rightarrow H^{1}_\Delta(\Omega)$ is an isomorphism.

Now we are led to the equivalent optimization problem
\begin{eqnarray}
\min\limits_{y_u\in K}\ J(y_u)={1\over 2}\|y_f+y_u-y_d\|_{0,\Omega}^2+{\alpha\over 2}\|\nabla y_u\|_{0,\Omega}^2\nonumber\\
\mbox{subject\ to}\quad y_u\in H^1(\Omega):\quad (\nabla y_u,\nabla v)=0\quad\forall v\in H_0^1(\Omega),\label{Dirichlet_VI}
\end{eqnarray}
where $K:=\{v\in H^1(\Omega):\ v\leq y_b-y_f\ \mbox{a.e.\ in}\ \Omega\}$.

Define the kernel space 
\begin{eqnarray}
V_D^0=\{y\in H^1(\Omega):\quad (\nabla y,\nabla v)=0\quad\forall v\in H_0^1(\Omega) \}.\nonumber 
\end{eqnarray}
Then the above optimization problem can be written as 
\begin{eqnarray}
\min\limits_{y_u\in V_D^0\cap K}\ J(y_u)={1\over 2}\|y_f+y_u-y_d\|_{0,\Omega}^2+{\alpha\over 2}\|\nabla y_u\|_{0,\Omega}^2,\nonumber
\end{eqnarray}
which can be characterized as: Find $y_u\in V_D^0\cap K$ such that
\begin{eqnarray}
\alpha (\nabla y_u,\nabla(v-y_u))+(y_u,v-y_u)\geq (y_d-y_f,v-y_u)\quad \forall v\in V_D^0\cap K.\nonumber
\end{eqnarray}

\subsection{Neumann boundary control problems with pointwise state constraints}

In this subsection we consider the Neumann boundary control problem with pointwise state constraints
\begin{eqnarray}
\min\limits_{u\in U}\quad J(y,u)={1\over 2}\|y-y_d\|_{0,\Omega}^2+{\alpha\over 2}\|u\|_{U}^2
\end{eqnarray}
subject to 
\begin{equation}\label{state_Neumann}
\left\{
\begin{aligned}
Ly:=-\Delta y+y=f\quad&\mbox{in}\ \Omega,\\
\partial_ny=u\quad&\mbox{on}\ \Gamma,\\
y\leq y_b\quad&\mbox{in}\ \Omega.
\end{aligned}
\right.
\end{equation}
Similar to the Dirichlet case, there are also several choices for the control space $U$. The first and the natural one, is that $U:=L^2(\Gamma)$ (cf. \cite{ApelPfeffererRosch,CasasMateos,CasasMateosTroltzsch,HinzeMatthes}). With this choice we have that the associated state variable $y:=Su\in H^{3/2}(\Omega)$ because $S:L^2(\Gamma)\rightarrow H^{3/2}(\Omega)$ is well-defined.  Note that in dimension two we have $H^{3/2}(\Omega)\hookrightarrow C(\bar\Omega)$. Therefore, the pointwise state constraints can be understood in $C(\bar\Omega)$. In general, to ensure that $y\in C(\bar\Omega)$ one has to choose the control space $L^t(\Gamma)$ for some $t>d-1$ where $d$ is the dimension of $\Omega$ (cf. \cite{Casas1997,Casas}). Another approach is to impose box type control constraints to ensure that $u\in L^\infty(\Gamma)$. In  \cite{KrumbiegelMeyerRosch} the authors imposed the pointwise state constraints in a subdomain $\Omega_0\subset\subset\Omega$ by using the higher interior regularity of elliptic equations.

The second choice is the so-called energy space, i.e., $U=H^{-1/2}(\Gamma)$ (cf. \cite{ApelSteinbachWinkler}). Then one can only expect that $y:=Su\in H^{1}(\Omega)$ because $S:H^{-1/2}(\Gamma)\rightarrow H^{1}_\Delta(\Omega):=\{v\in H^1(\Omega), Lv=0\ \mbox{in}\ \Omega\}$ is well-defined. 

Now we consider the control space $U=H^{-1/2}(\Gamma)$. Recall that we have the equivalent representation of the $H^{-1/2}(\Gamma)$-norm (cf. \cite{ApelSteinbachWinkler}):
\begin{eqnarray}
\|u\|_{H^{-1/2}(\Gamma)}^2:=\langle u,\mathcal{N}u\rangle_{H^{-1/2}(\Gamma);H^{1/2}(\Gamma)}=(\nabla y_u,\nabla y_u)+(y_u,y_u)=\|y_u\|_{1,\Omega}^2,\nonumber
\end{eqnarray}
where $\mathcal{N}u:=\gamma_0y_u$ is the Neumann-to-Dirichlet operator with $\gamma_0:H^1(\Omega)\rightarrow H^{1/2}(\Gamma)$ the Dirichlet trace operator and $y_u\in H^1(\Omega)$ satisfies the following Neumann boundary value problem with homogeneous right-hand side
\begin{equation}\label{state_Neumann_u}
\left\{
\begin{aligned}
-\Delta y_u+y_u=0\quad&\mbox{in}\ \Omega,\\
\partial_ny_u=u\quad&\mbox{on}\ \Gamma.
\end{aligned}
\right.
\end{equation}
In fact, an integration by parts yields the result. Let $y_f\in H_0^1(\Omega)$ solve
\begin{equation}\label{state_Neumann_f}
\left\{
\begin{aligned}
-\Delta y_f+y_f=f\quad&\mbox{in}\ \Omega,\\
\partial_ny_f=0\quad&\mbox{on}\ \Gamma.
\end{aligned}
\right.
\end{equation}
Then we have $y=y_f+y_u$. Note that $S:H^{-1/2}(\Gamma)\rightarrow H^{1}_\Delta(\Omega)$ is an isomorphism. 

Now we are led to the equivalent  optimization problem
\begin{eqnarray}
\min\limits_{y_u\in K}\ J(y_u)={1\over 2}\|y_f+y_u-y_d\|_{0,\Omega}^2+{\alpha\over 2}\|y_u\|_{1,\Omega}^2\nonumber\\
\mbox{subject\ to}\quad y_u\in H^1(\Omega):\quad a(y_u,v)=0\quad\forall v\in H_0^1(\Omega).\label{Neumann_VI}
\end{eqnarray}

Define the kernel space 
\begin{eqnarray}
V_N^0=\{y\in H^1(\Omega):\quad (\nabla y,\nabla v)+(y,v)=0\quad\forall v\in H_0^1(\Omega) \}.\nonumber 
\end{eqnarray}
Then the above optimization problem can be written as 
\begin{eqnarray}
\min\limits_{y_u\in V_N^0\cap K}\ J(y_u)={1\over 2}\|y_f+y_u-y_d\|_{0,\Omega}^2+{\alpha\over 2}\|y_u\|_{1,\Omega}^2,\nonumber
\end{eqnarray}
which can be characterized as: Find $y_u\in V_N^0\cap K$ such that
\begin{eqnarray}
\alpha a(y_u,v-y_u)+(y_u,v-y_u)\geq (y_d-y_f,v-y_u)\quad \forall v\in V_N^0\cap K.\nonumber
\end{eqnarray}

\subsection{A general variational inequality with equality constraint and its numerical solution}
The variational inequality with equality constraint arising from either the Dirichlet or Neumann boundary control problem can be studied in a unified framework. 

We consider a general  minimization problem
\begin{eqnarray}
\min\limits_{y_u\in V^0\cap K}\ J(y_u)={\alpha\over 2}a(y_u,y_u)+{1\over {2}}\|y_u+y_f-y_d\|_{0,\Omega}^2,
\end{eqnarray}
where $V^0:=\{y\in H^1(\Omega): \ a(y,v)=0\ \forall v\in H_0^1(\Omega)\}$.

\begin{Remark}
Note that $V^0$ is a closed subspace of $H^1(\Omega)$. Following similar arguments as those of Example 1.7 in \cite{Ito_2008}, there exists a Lagrange multiplier $\mu_0\in (V^0)^*$ for the above optimization problem. Since $V^0$ is closed, $\mu_0$ can be extended as an element in $H^1(\Omega)^*$ by the Hahn-Banach theorem.
\end{Remark}

Introduce the Lagrangian functional
\begin{eqnarray}
\mathcal{L}(y_u,p)=J(y_u)-a(y_u,p),\quad p\in H_0^1(\Omega).
\end{eqnarray}
Differentiating $\mathcal{L}(y_u,p)$ with respect to $y_u$ yields
\begin{eqnarray}
\alpha a(y_u,v-y_u)+(y_u,v-y_u)\geq (y_d-y_f,v-y_u)+a(p,v-y_u)\quad\forall v\in K.\nonumber
\end{eqnarray}
Differentiating $\mathcal{L}(y_u,p)$ with respect to $p$ gives
\begin{eqnarray}
a(y_u,v)=0\quad\forall v\in H_0^1(\Omega).\nonumber
\end{eqnarray}
Then we are led to a general second order elliptic variational inequality: Find $(y_u,p)\in K\times H_0^1(\Omega)$ such that 
\begin{eqnarray}
a(y_u,v)=0\quad\forall v\in H_0^1(\Omega),\nonumber\\
\alpha a(y_u,(w-y_u))+(y_u,w-y_u)\geq (y_d-y_f,w-y_u)+a(p,w-y_u)\quad\forall w\in K.\nonumber
\end{eqnarray}
From the second equation, there exists a $\lambda\in H^1(\Omega)^*$ such that
\begin{eqnarray}
\alpha a(y_u,w)+(y_u,w) - \langle\lambda,w\rangle_{H^1(\Omega)^*,H^1(\Omega)}= (y_d-y_f,w)+a(p,w)\quad\forall w\in H^1(\Omega)\nonumber
\end{eqnarray}
and $\lambda \leq 0$, $\langle\lambda,y_u-y_b+y_f\rangle_{H^1(\Omega)^*,H^1(\Omega)}=0$. 

In the following we give a further look at the above system. Let $v=y_u+\psi$ with $\psi\in\mathcal{D}(\Omega)$ and $\psi\leq 0$. It is clear that $v\in H^1(\Omega)$  and $v\leq y_b-y_f$. Moreover, $v-y_u=\psi\in H_0^1(\Omega)$. Substituting $v$ into the variational inequality and recalling that $y_u$ is harmonic, we have
\begin{eqnarray}
a(p,\psi)\leq (y_f+y_u-y_d,\psi)\nonumber
\end{eqnarray}
or equivalently, 
\begin{eqnarray}
\int_\Omega (Lp-y+y_d)\psi dx\leq 0\quad\forall \psi\in \mathcal{D}(\Omega),\ \psi\leq 0.\nonumber
\end{eqnarray}
That is,
\begin{eqnarray}
L p-y+y_d\geq 0\quad\mbox{a.e.\ in}\ \Omega. \label{positive}
\end{eqnarray}
Now we define the active and inactive sets
\begin{eqnarray}
\Omega^-:=\{x\in \Omega: y_u(x)<(y_b-y_f)(x)\},
\quad\Omega^0:=\{x\in \Omega: y_u(x)=(y_b-y_f)(x)\}.\nonumber
\end{eqnarray}
Then for any $\theta \in \mathcal{D}(\Omega)$, $\theta(x)\in [0,1]$, we have $v=\theta (y_b-y_f)(x)+(1-\theta)y_u\in H^1(\Omega)$ and $v\leq y_b-y_f$. Substituting $v$ into the variational inequality and recalling that $y_u$ is harmonic, we have
\begin{eqnarray}
\int_\Omega (L p-y+y_d)\theta (y_b-y)dx\leq 0\quad\forall \theta \in \mathcal{D}(\Omega),\ \theta(x)\in [0,1].\nonumber
\end{eqnarray}
That is,
\begin{eqnarray}
Lp-y+y_d\leq 0\quad\mbox{a.e.\ in}\ \Omega^-. 
\end{eqnarray}
Combining the above results we have
\begin{eqnarray}
Lp=y-y_d\quad\mbox{a.e.\ in}\ \Omega^-\quad \mbox{and}\quad L p-y+y_d\geq 0\quad\mbox{a.e.\ in}\ \Omega^0.
\end{eqnarray}

There are several approaches to solve the above variational inequality with equality constraint, which are briefly reviewed below.
\begin{enumerate}[(1)]
\item Let
 \begin{center}$S=\left[
  \begin{array}{cc}
    \alpha L+I &-L\\
     L & 0  \\
  \end{array}
\right]$,\ \  $b = \left[
     \begin{array}{cc}
     y_d-y_f\\
     0\\
     \end{array}
     \right]$,\ \ $X = \left[
     \begin{array}{cc}
     y\\
     p\\
     \end{array}
     \right]$.
\end{center}
Then the coupled system can be written as the following linear projection operator equation
\begin{eqnarray}
X=P_{X_{ad}}(X-(SX+b)),\nonumber
\end{eqnarray}
where $P_{X_{ad}}$ denotes the orthogonal projection onto the admissible set $X_{ad}:=\{(v_1,v_2)\in H^1(\Omega)\times H_0^1(\Omega):\ v_1\leq y_b-y_f\quad \mbox{in}\ \Omega\}$. We can solve the above linear projection equation on the discrete level by using the method proposed in \cite{He}. 

\item  The discrete problem can be solved by the augmented Lagrangian method. Define 
\begin{eqnarray}
\mathcal{L}_\lambda(y_u,p)=J(y_u)+(L y_u,p)+\lambda\|L y_u\|^2_{0,\Omega}.\nonumber
\end{eqnarray}
Then we solve sequentially $\min\limits_{y_u\in H^2(\Omega),y_u\leq y_b-y_f}\mathcal{L}_\lambda(y_u,p)$ for given $p$ and $\max\limits_{p\in L^2(\Omega)}\mathcal{L}_\lambda(y_u,p)$ for given $y_u$. We remark that with this approach we need to solve fourth order variational inequalities. 

\item The discrete problem can be solved by the primal-dual hybrid gradient (PDHG) method. Given $(y_u^k,p^k)$ and $\gamma, s > 0$ sufficiently large, we solve sequentially
\begin{enumerate}[$\bullet$]
\item $y_u^{k+1}=\argmin\limits_{y_u\in H^1(\Omega)}\ \{\mathcal{L}(y_u,p^k)+{\gamma\over 2}a(y_u-y_u^k,y_u-y_u^k):\ y_u\leq y_b-y_f\}$. That is, find $y_u\in H^1(\Omega)$ such that $y_u\leq y_b-y_f$ and
\begin{eqnarray}
(\alpha+\gamma) a(y_u,v-y_u)+(y_u,v-y_u) \geq \gamma a(y_u^k,v-y_u)&+(y_d-y_f,v-y_u)\nonumber\\
&+a(p^k,v-y_u).\nonumber
\end{eqnarray}

\item $p^{k+1}=\argmax\limits_{p\in H_0^1(\Omega)}\ \{\mathcal{L}(2y_u^{k+1}-y_u^k, p)-{s\over 2}\|\nabla(p-p^k)\|_{0,\Omega}^2$, which can be characterized by
\begin{eqnarray}
s(\nabla p^{k+1}, \nabla v)=s(\nabla p^k, \nabla v)-a(2y_u^{k+1}-y_u^k, v)\quad\forall v\in H_0^1(\Omega).\nonumber
\end{eqnarray}
\end{enumerate}
\end{enumerate}

\subsection{Extensions and comments}
In the above several subsections we consider the Laplace operator. We remark that the results can be readily extended to general second order elliptic operators. For example, we can consider the following second order elliptic operator:
\begin{eqnarray}
Ly:=-\sum\limits_{i,j=1}^d\partial_{x_i}(a_{ij}\partial_{x_j}y)+a_0y,
\end{eqnarray}
where $a_{ij}\in L^\infty(\Omega)$, $i,j=1,\cdots, d$, $a_0> 0$. The coefficient matrix $A=(a_{ij})_{i,j=1}^d$ is further assumed to be symmetric and positive  definite. Then we can define the equivalent definition of the $H^{1/2}(\Gamma)$-norm and $H^{-1/2}(\Gamma)$-norm as follows:
\begin{eqnarray}
&\|u\|_{H^{-1/2}(\Gamma)}^2:=(A\nabla y_u,\nabla y_u)+a_0(y_u,y_u),\nonumber\\
&\|u\|_{H^{1/2}(\Gamma)}^2:=(A\nabla w_u,\nabla w_u)+a_0(w_u,w_u),\nonumber
\end{eqnarray}
where $y_u$ and $w_u$ satisfy respectively
\begin{equation}\label{state}
\left\{
\begin{aligned}
Ly_u=0\quad&\mbox{in}\ \Omega,\\
A\nabla y_u\cdot{\bf n}=u\quad&\mbox{on}\ \Gamma
\end{aligned}
\right.
\end{equation}
and
\begin{equation}\label{state}
\left\{
\begin{aligned}
Lw_u=0\quad&\mbox{in}\ \Omega,\\
w_u=u\quad&\mbox{on}\ \Gamma.
\end{aligned}
\right.
\end{equation}
Proceeding as above we can obtain the similar second order elliptic variational inequality. 

Of course, we can also consider the Dirichlet or Neumann boundary control problems with pointwise gradient state constraints.

For PDE-constrained optimal control problems we need to impose pointwise control constraints in certain applications. This is well-studied for problems posed in $L^2(\Omega)$ for distributed controls and in $L^2(\Gamma)$ or $H^{1/2}(\Gamma)$ for boundary control problems. However, for distributed and Neumann boundary control problems posed in energy spaces these control constraints are not meaningful in the classical sense because the control functions are only distributions, and can only be understood in the duality sense (cf. \cite{ApelSteinbachWinkler} for the Neumann case). On the other hand, for Dirichlet boundary control problems the pointwise state constraints impose automatically constraints on the control variable as the Dirichlet trace of the state.   
 
\section{Finite element approximations and convergence analysis}
\setcounter{equation}{0}
In this section we consider the finite element approximations to the resulted variational inequalities. Let $\{\mathcal{T}_h\}_{h>0}$ be a family of  shape regular triangulations of the domain $\Omega$. Associated with $\mathcal{T}_h$ we construct the piecewise linear and continuous finite element space $V_h$. Denote by $V_h^0=V_h\cap H_0^1(\Omega)$. Let $y_{b,h}\in V_h$ be an approximation of $y_b$, which can be chosen as the Lagrange interpolation or $L^2$-projection of $y_b$, depending on the smoothness of $y_b$. 

Let $y_{f,h}\in V_h^0$ satisfy
\begin{eqnarray}
a(y_{f,h},v_h)=(f,v_h)\quad\forall v_h\in V_h^0\label{y_f_h_D}
\end{eqnarray}
for Dirichlet boundary value problems and $y_{f,h}\in V_h$ satisfy
\begin{eqnarray}
a(y_{f,h},v_h)=(f,v_h)\quad\forall v_h\in V_h\label{y_f_h_N}
\end{eqnarray}
for Neumann boundary value problems.

\subsection{Distributed control problems}
In this subsection we consider the finite element approximation to the variational inequality resulted from the distributed control problem. Define 
\begin{eqnarray}
K_h^0:=\{v_h\in V_h^0:\quad v_h\leq \tilde{y}_{b,h}:=y_{b,h}-y_{f,h}\}.
\end{eqnarray}

Now we consider the finite element approximation to the variational inequality (\ref{VI_distributed}). The discrete scheme reads: Find $y_{u,h}\in K_h^0$ such that
\begin{eqnarray}
\alpha a( y_{u,h},v_h-y_{u,h})+(y_{u,h},v_h-y_{u,h})\geq (y_d-y_{f,h},v_h-y_{u,h})\quad\forall v_h\in K_h^0.\label{Distribute_VI_discrete}
\end{eqnarray} 
Then it is clear that $y_h:=y_{u,h}+y_{f,h}\in V_h^0$ is an approximation to $y$ and there holds $y_h\leq y_{b,h}$ in $\Omega$.

Once we obtain the discrete state $y_h$, we can recover the distributed control $u_h$ by solving the following equation: Find $u_h\in V_h^0$ such that
\begin{eqnarray}
(u_h,v_h)=(\nabla y_h,\nabla v_h)-(f,v_h)\quad\forall v_h\in V_h^0.
\end{eqnarray}
This is motivated from the variational formulation of the state equation: Find $y\in H^1_0(\Omega)$ such that
\begin{eqnarray}
(\nabla y,\nabla v)=(f,v)+( u,v)\quad\forall v\in H^1_0(\Omega).
\end{eqnarray}
The finite element approximation to the variational inequality with gradient constraints can be formulated in a similar way by defining the discrete constraint set $K_h:=\{v_h\in V_h^0:\ |\nabla v_h|\leq y_{b,h}\}$.

Now we discuss the regularity of solutions to the above second order elliptic variational inequality. It was shown in \cite{KinderlehrerStampacchia} that if $Y=H_0^1(\Omega)$, $y_d\in L^2(\Omega)$ and $y_b\in H^2(\Omega)$ such that $y_b\geq 0$ on $\Gamma$, $\Omega$ is a convex polygonal or smooth domain, then $y\in H^2(\Omega)$.  Furthermore, it was stated in \cite{BrezziHagerRaviart} that if $y_d\in L^\infty(\Omega)\cap BV(\Omega)$, $y_b\in C^3(\Omega)$ such that $y_b>0$, and $\Gamma$ is sufficiently smooth, then $y\in W^{s,p}(\Omega)$ for all $1<p<\infty$ and $s<2+1/p$. 

\begin{Theorem}\label{Thm:priori}
Let $y=y_u+y_f$ and $y_h=y_{u,h}+y_{f,h}$, where $y_u$ and $y_{u,h}$ are the solutions of the second order elliptic variational inequality (\ref{VI_distributed}) and its piecewise linear finite element approximation (\ref{Distribute_VI_discrete}), $y_f$ and $y_{f,h}$ are the solutions of (\ref{state_Dirichlet_f}) and (\ref{y_f_h_D}), respectively. Assume that $y_u, y_f\in H^{1+\beta}(\Omega)$ for $\beta\in (0,1]$. Then we have the following an a priori error estimate
\begin{eqnarray}
\|y-y_h\|_{1,\Omega}\leq Ch^{\beta/2}(\|y\|_{1+\beta,\Omega}+\|y_b\|_{2,\Omega}).
\end{eqnarray}
Moreover, if $y_u, y_f\in H^{2}(\Omega)$, we have the following improved an a priori error estimate
\begin{eqnarray}
\|y-y_h\|_{1,\Omega}\leq Ch(\|y\|_{2,\Omega}+\|y_b\|_{2,\Omega}).
\end{eqnarray}

\end{Theorem}
\begin{proof}
The proof follows from the idea of \cite{BrezziHagerRaviart}. Here for completeness we include a proof.

Since $y=y_u+y_f$ and $y_h=y_{u,h}+y_{f,h}$, we have
\begin{eqnarray}
\|y-y_h\|_{1,\Omega}&\leq& \|y_u-y_{u,h}\|_{1,\Omega}+\|y_f-y_{f,h}\|_{1,\Omega}\nonumber\\
&\leq&\|y_u-y_{u,h}\|_{1,\Omega}+Ch^\beta\|y_f\|_{1+\beta,\Omega},\nonumber
\end{eqnarray}
where we used the standard error estimate for elliptic equations. It suffices to estimate $\|y_u-y_{u,h}\|_{1,\Omega}$. Let $b(y_u,v):=\alpha (\nabla y_u,\nabla v)+ (y_u,v)$. For any $v_h\in K_h\subset K:=\{v\in H_0^1(\Omega),\ v(x)\leq (y_b-y_{f})(x)\ \mbox{a.e.}\ x\in\Omega \}$, we have
\begin{eqnarray}
c\|y_u-y_{u,h}\|_{1,\Omega}^2&\leq &b(y_u-y_{u,h},y_u-y_{u,h})\nonumber\\
&=&b(y_u-y_{u,h},y_u-v_h)+b(y_u-y_{u,h},v_h-y_{u,h})\nonumber\\
&=&b(y_u-y_{u,h},y_u-v_h)+\langle\mu, v_h-y_{u,h}\rangle_{H^{-1}(\Omega),H_0^1(\Omega)}\nonumber\\
&&+(y_d-y_{f},v_h-y_{u,h}) - b(y_{u,h},v_h-y_{u,h})\nonumber\\
&\leq&b(y_u-y_{u,h},y_u-v_h)+\langle\mu, v_h-y_{u,h}\rangle_{H^{-1}(\Omega),H_0^1(\Omega)}\nonumber\\
&&+(y_{f,h}-y_f,v_h-y_{u,h}),\nonumber
\end{eqnarray}
where we used (\ref{Distribute_VI_discrete}) in the last inequality. 

Let $v_h:=y^I_u\in K_h$ be the Lagrange interpolation of $y_u$. We can derive from the standard interpolation error estimate that
\begin{eqnarray}
c\|y_u-y_{u,h}\|_{1,\Omega}^2&\leq&b(y_u-y_{u,h},y_u-y^I_u)+\langle\mu, y^I_u-y_{u,h}\rangle_{H^{-1}(\Omega),H_0^1(\Omega)}\nonumber\\
&&+(y_{f,h}-y_f,y^I_u-y_{u,h})\nonumber\\
&\leq&Ch^\beta\|y_u-y_{u,h}\|_{1,\Omega}\|y_u\|_{1+\beta,\Omega}+Ch^{2\beta}\|y_f\|_{1+\beta,\Omega}\|y^I_u-y_{u,h}\|_{0,\Omega} \nonumber\\
&&+ \langle\mu, y^I_u-y_{u,h}\rangle_{H^{-1}(\Omega),H_0^1(\Omega)}.\nonumber
\end{eqnarray}
Since
$$\|y^I_u-y_{u,h}\|_{0,\Omega} \leq  \|y^I_u-y_{u}\|_{0,\Omega}  + \|y_u-y_{u,h}\|_{0,\Omega}\leq  Ch^{1+\beta}\|y_u\|_{1+\beta,\Omega} + C\|y_u-y_{u,h}\|_{1,\Omega}, $$
by Young's inequality, we have
$$\|y_u-y_{u,h}\|_{1,\Omega}^2 \leq  Ch^{2\beta
} + \langle\mu, y^I_u-y_{u,h}\rangle_{H^{-1}(\Omega),H_0^1(\Omega)}.$$

Now it remains to estimate the second term. Note that there holds 
\begin{eqnarray}
\mu\leq 0,\quad \langle\mu, y_b-y_f-y_u\rangle_{H^{-1}(\Omega),H_0^1(\Omega)}=0 \nonumber.
\end{eqnarray}
Therefore, from $y_u\leq y_b-y_f$ and $y_{u,h}\leq y_{b,h}-y_{f,h} = y_b^I - y_{f,h}$ we have
\begin{eqnarray}
&&\langle\mu, y^I_u-y_{u,h}\rangle_{H^{-1}(\Omega),H_0^1(\Omega)}\nonumber\\
&=&\langle\mu,y^I_u-y_b^I+y_f^I-(y_u-y_b+y_f)\rangle_{H^{-1}(\Omega),H_0^1(\Omega)} \nonumber\\
&&+ \langle\mu, y_u-y_b+y_f\rangle_{H^{-1}(\Omega),H_0^1(\Omega)} +\langle\mu,y_b^I-y_f^I-y_{u,h}\rangle_{H^{-1}(\Omega),H_0^1(\Omega)}\nonumber\\
&\leq&\langle\mu,y^I_u-y_b^I+y_f^I-(y_u-y_b+y_f)\rangle_{H^{-1}(\Omega),H_0^1(\Omega)}+\langle\mu,y_{f,h}-y_f^I\rangle_{H^{-1}(\Omega),H_0^1(\Omega)}\nonumber\\
&\leq & Ch^{\beta}\|\mu\|_{-1,\Omega}(\|y_u-y_b+y_f\|_{1+\beta,\Omega} + \|y_f\|_{1+\beta,\Omega})\nonumber.
\end{eqnarray}

If $y_u, y_f\in H^2(\Omega)$, we have $\mu\in L^2(\Omega)$ by Remark \ref{rem:dis_LM_regularity}. Then
\begin{eqnarray}
\langle\mu, y^I_u-y_{u,h}\rangle_{H^{-1}(\Omega),H_0^1(\Omega)}
&\leq& (\mu,y^I_u-y_b^I+y_f^I-(y_u-y_b+y_f))+(\mu,y_{f,h}-y_f^I)\nonumber\\
&\leq&Ch^2\|\mu\|_{0,\Omega}(\|y_u-y_b+y_f\|_{2,\Omega}+\|y_f\|_{2,\Omega}).\nonumber
\end{eqnarray}

Hence,
\begin{equation}\nonumber
\|y_u-y_{u,h}\|_{1,\Omega} \leq
\left\{
\begin{aligned}
&Ch^{\beta/2}\quad \mbox{if}\ \ \beta\in (0,1),\\
&Ch\quad\ \ \ \ \ \mbox{if}\ \ \beta =1\\
\end{aligned}
\right.
\end{equation}
and
\begin{equation}\nonumber
\|y-y_{h}\|_{1,\Omega} \leq
\left\{
\begin{aligned}
&Ch^{\beta/2}\quad \mbox{if}\ \ \beta\in (0,1),\\
&Ch\quad\ \ \ \ \ \mbox{if}\ \ \beta =1.\\
\end{aligned}
\right.
\end{equation}
This completes the proof.
\end{proof}

\begin{Remark}
For $y_u, y_f\in H^{1+\beta}(\Omega)$ with $\beta\in (0,1]$, if we can show $\mu\in (H^{1-\beta}(\Omega))^*$, then we have
$$ \|y-y_{h}\|_{1,\Omega} \leq Ch^\beta.$$
\end{Remark}

\subsection{Dirichlet boundary control problems}
In this subsection we consider the finite element approximation to the variational inequality resulted from Dirichlet boundary control problems. 

Define 
\begin{eqnarray}
K_h:=\{v_h\in V_h:\quad v_h\leq \tilde{y}_{b,h}:=y_{b,h}-y_{f,h}\}.\nonumber
\end{eqnarray}
Now we consider the finite element approximation to the variational inequality (\ref{Dirichlet_VI})\begin{eqnarray}
\min\limits_{y_{u,h}\in K_h}\ J(y_{u,h})={1\over 2}\|y_{f,h}+y_{u,h}-y_d\|_{0,\Omega}^2+{\alpha\over 2}\|\nabla y_{u,h}\|_{0,\Omega}^2\nonumber\\
\mbox{subject\ to}\quad y_{u,h}\in V_h:\quad (\nabla y_{u,h},\nabla v_h)=0\quad\forall v_h\in V_h^0.\label{Dirichlet_VI_h}
\end{eqnarray}

Define the discrete kernel space 
\begin{eqnarray}
V_{D,h}^0=\{y_h\in V_h:\quad (\nabla y_h,\nabla v_h)=0\quad\forall v_h\in V_h^0\}.\nonumber
\end{eqnarray}
Then the above optimization problem can be written as 
\begin{eqnarray}
\min\limits_{y_{u,h}\in V_{D,h}^0\cap K_h}\ J(y_{u,h})={1\over 2}\|y_{f,h}+y_{u,h}-y_d\|_{0,\Omega}^2+{\alpha\over 2}\|\nabla y_{u,h}\|_{0,\Omega}^2,\nonumber
\end{eqnarray}
which can be characterized as: Find $y_{u,h}\in V_{D,h}^0\cap K_h$ such that 
\begin{eqnarray}
\alpha a(y_{u,h},v_h-y_{u,h})+(y_{u,h},v_h-y_{u,h})\geq (y_d-y_{f,h},v_h-y_{u,h})\quad\forall v_h\in V_{D,h}^0\cap K_h.\nonumber
\end{eqnarray}
After solving the above variational inequality, we then obtain an approximation $y_h:=y_{u,h}+y_{f,h}\in V_h$ to $y$ and there holds $y_h\leq y_{b,h}$ in $\Omega$. Once the discrete state $y_h$ is obtained, we can recover the discrete control variable $u_h:=y_h|_{\Gamma}\in V_h(\Gamma)$.

By using the Lagrangian multiplier method, the discrete optimization problem is equivalent to find $y_{u,h}\in K_h$ and $p_h\in V_h^0$ such that
\begin{eqnarray}
a(y_{u,h},v_h)=0\quad\forall v_h\in V_h^0,\nonumber\\
\alpha a(y_{u,h},w_h-y_{u,h})+(y_{u,h},w_h-y_{u,h})\geq (y_d-y_{f,h},w_h-y_{u,h})+a(p_h,w_h-y_{u,h})\nonumber
\end{eqnarray}
for $\forall w_h\in K_h$.

\subsection{Neumann boundary control problems}
In this subsection we consider the finite element approximation to the variational inequality resulted from Neumann boundary control problems. 

We define the following discrete optimization problem 
\begin{eqnarray}
\min\limits_{y_{u,h}\in K_h}\ J(y_{u,h})={1\over 2}\|y_{f,h}+y_{u,h}-y_d\|_{0,\Omega}^2+{\alpha\over 2}\|y_{u,h}\|_{1,\Omega}^2\nonumber\\
\mbox{subject\ to}\quad y_{u,h}\in V_h:\quad a(y_{u,h},v_h)=0\quad\forall v_h\in V_h^0.\label{Neumann_VI_h}
\end{eqnarray}

Define the discrete kernel space 
\begin{eqnarray}
V_{N,h}^0=\{y_h\in V_h:\quad a(y_h,v_h)=0\quad\forall v_h\in V_h^0\}.\nonumber
\end{eqnarray}
Then the above optimization problem can be written as 
\begin{eqnarray}
\min\limits_{y_{u,h}\in V_{N,h}^0\cap K_h}\ J(y_{u,h})={1\over 2}\|y_{f,h}+y_{u,h}-y_d\|_{0,\Omega}^2+{\alpha\over 2}\|y_{u,h}\|_{1,\Omega}^2,\nonumber
\end{eqnarray}
which can be characterized as: Find $y_{u,h}\in V_{N,h}^0\cap K_h$ such that
\begin{eqnarray}
\alpha a(y_{u,h},v_h-y_{u,h})+(y_{u,h},v_h-y_{u,h})\geq (y_d-y_{f,h},v_h-y_{u,h})\quad\forall v_h\in V_{N,h}^0\cap K_h.\nonumber
\end{eqnarray}
After solving the above variational inequality, we then obtain an approximation $y_h:=y_{u,h}+y_{f,h}\in V_h$ to $y$ and there holds $y_h\leq y_{b,h}$ in $\Omega$.

Once we obtain the discrete state $y_h$, we can recover the boundary control $u_h$ by solving the following equation: Find $u_h\in V_h(\Gamma):=V_h|_{\Gamma}$ such that
\begin{eqnarray}
\langle u_h,v_h\rangle_\Gamma=a(y_h,v_h)-(f,v_h)\quad\forall v_h\in V_h.\label{Neumann_uh}
\end{eqnarray}
This is motivated from the variational form of the state equation: Find $y\in H^1(\Omega)$ such that
\begin{eqnarray}
a(y,v)=(f,v)+\langle u,v\rangle_\Gamma\quad\forall v\in H^1(\Omega).
\end{eqnarray}
Another trivial approach is to compute $u_h=\nabla y_h\cdot n|_{\Gamma}$.

At this moment we are not able to derive a priori error estimates for the finite element approximation to the variational inequality derived from Dirichlet or Neumann boundary control problems. The main reason is that we do not have a precise characterization of the Lagrange multiplier for the second order elliptic inequality with equality constraints. However,  numerical results in Section 5 indicate a first order convergence for the state under the energy norm. The theoretical validation remains open.

\section{Numerical results}
\setcounter{equation}{0}

In this section we carry out extensive numerical experiments to show the numerical performance of the proposed algorithms. We consider the Dirichlet and Neumann boundary control problems with pointwise state constraints, as well as the distributed control problems with pointwise constraints on either the state or the state gradient.

Denote by $p_i$ the nodes and $\psi_i$ the corresponding linear
nodal basis functions, i.e. $\psi_i(p_j) = \delta_{ij}$, where
$\delta_{ij}$ is the Chronecker symbol. Then the finite element
approximations of the adjoint state and the state are $p_h = \sum\limits_{i=1}^{N_0} P_i\psi_i$ and
$y_{u,h}=\sum\limits_{i=1}^N Y_i\psi_i$, where $N$ and $N_0$ are the numbers of degrees of freedom for $V_h$ and $V_h^0$ respectively. Let
\begin{eqnarray}
(L_h)_{i,j} = a(\psi_i,\psi_j),\ \  B_{i,j} =
(\psi_i,\psi_j),\ \ 
f_{i} = (y_d-y_f,\psi_i),\ \  \Psi_i =
y_b(p_i),\label{6.2}
\end{eqnarray}
where $y_d$ is the desired state and $y_b$ is the obstacle.

The discrete variational inequality (\ref{Dirichlet_VI_h}) or (\ref{Neumann_VI_h}) with equality constraint can be formulated and solved by (\ref{Projection_equ}) where 
 \begin{center}$S=\left[
   \begin{array}{cc}
    (\alpha L_h+B)_{N\times N} &-(L_h)_{N\times N_0}\\
     (L_h)_{N_0\times N} & 0_{N_0\times N_0}  \\
  \end{array}
\right]$,\ \  $b = \left[
     \begin{array}{cc}
     f\\
     0\\
     \end{array}
     \right]$,\ \ $X = \left[
     \begin{array}{cc}
     Y\\
     P\\
     \end{array}
     \right]$.
\end{center}
Then the coupled system can be written as the following linear projection equation
\begin{eqnarray}
X=P_{X_{ad}}(X-(SX+b)),\label{Projection_equ}
\end{eqnarray}
where $P_{X_{ad}}$ denotes the orthogonal projection onto the admissible set $X_{ad}:=\{(v_1,v_2)\in H^1(\Omega)\times H_0^1(\Omega):\ v_1\leq y_b-y_f\quad \mbox{in}\ \Omega\}$. 

To solve the above projection equation, we use a projection
gradient method that was proposed in \cite{He}. The variational inequality (\ref{Distribute_VI_discrete}) arising from the distribute control can also be solved by the mentioned projection algorithm, other optimization algorithms (cf. \cite{GlowinskiLions}), e.g., the semi-smooth Newton method or SOR method, can also be applied.  The projection gradient algorithm can be described as follows:
\\~\\
\fbox{\parbox{\textwidth}{ \begin{Algorithm} \label{Alg:projection}
Given $S$, $b$ and
tolerance TOL;

 (1) $g = SX + b$;
 
 (2) $e = X - \mbox{proj}\ (X-g)$;

 (3) error $= \|e\|_2$, if error $<$ TOL then stop and
output X;

 (4) $d=S^Te+(SX + b)$;

 (5) $\rho=\frac{\|e\|_2^2}{\|(S^T+I)e\|_2^2}$;

(6) $X = \mbox{proj}\ (X - \rho d)$, go to (2).

\end{Algorithm}
}} \vskip 0.5cm

In Algorithm \ref{Alg:projection}, $\mbox{proj} (\cdot)$ denotes the projection
operator. In our following tests we set ${\rm TOL}=5.0e-8$. For the variational inequality with additional equality constraint resulted from Dirichlet or Neumann boundary control problems, we also test the PDHG method and obtain the similar results. However, PDHG is much slower than the above projection method so we omit the corresponding results.

\subsection{Elliptic distributed optimal control problems}

In the following numerical examples, we set $\Omega=(0,1)^2$, $f=0$ and consider the following four cases (cf. \cite{GongYan} for the first three cases)
\begin{itemize} 
\item Case 1: $\alpha = 1e-4$, $y_d(x_1, x_2)=\sin(4\pi x_1x_2) + 1.5$ and $y_b = 1$.

\item Case 2: $\alpha = 1e-3$, $y_d(x_1, x_2) = \sin(2\pi x_1x_2)$ and $y_b = 0.1$.

\item Case 3: $\alpha= 0.1$, $y_d(x_1, x_2) = 10(\sin(2x_1)+x_2)$ and $y_b = 0.01$.

\item Case 4: $\alpha= 0.1$, $y_d(x_1, x_2) = \sin(\pi x_1)\sin(\pi x_2)$ and $y_b = 0.1$.

\end{itemize}

In these cases we test the influence of the regularization parameter $\alpha$ and the consistency of the boundary conditions of the target state $y_d$. 

In order to test the performance of the algorithm, we also consider an example posed on a nonconvex polygonal domain. We set the following data and test this case for only the distributed control problem:
\begin{itemize}
\item Case 5: set $\Omega=(-1,1)^2\backslash \{[0,1]\times [-1,0]\}$, $\alpha =0.1$, $y_d=r^{-1/3}\tan (2/3\theta)$ and $y_b=0.5$, where $(r,\theta)$ is the polar coordinate.
\end{itemize}

\begin{Example}\label{Exm:Distributed}
We consider the elliptic distributed optimal control problems with pointwise state constraints with data given by Case 1, Case 2, Case 3, Case 4 and Case 5.
\end{Example}
We present the profiles of the discrete states and controls in Figure \ref{Fig:Distribute_Case1}-\ref{Fig:Distribute_Case4} for the first four cases. We can compare the results for the  first three cases with that computed in \cite{GongYan}. We can see a big difference for both the state and the control, and the discrete controls are generally not smooth. This is reasonable because the distributed control problem in $H^{-1}(\Omega)$ allows nonsmooth functions compared to the control problem posed in $L^2(\Omega)$. This is also observed in \cite{LangerSteinbachYang}. The boundary conditions of the target state for the first three cases are not consistent with the state equation, and we can observe oscillations of the discrete control near the boundary. However, for the fourth case with consistent target state, we do not observe this oscillation. Moreover, we can observe singular behaviors of the discrete controls along the boundary of active sets of the state variable.

We also test the convergence rates for the state and control approximations as shown in Table \ref{Tab:Distribute_Case1}-\ref{Tab:Distribute_Case4} for the first four cases. Since there are no analytical solutions, we use the computed discrete state and control on a fine mesh with 1050625 degrees of freedom for Case 1 and 263169 degrees of freedom for the other cases as the reference solutions. We can observe first order convergence for the energy norm error of the discrete state variable which confirms our theoretical results. We also compute the $L^2$-norm error of the discrete state and observe second order convergence rate, which is in agreement with many numerical observations for elliptic obstacle problems but is not rigorously proved. For the convergence of the discrete control, we can observe nearly half an order convergence. This suggests that the control is in $L^2(\Omega)$ which is consistent with the $H^2(\Omega)$-regularity of the state variable. However, a  function in $L^2(\Omega)$ can not merely admit $L^2(\Omega)$-smoothness but usually exhibits $H^{1/2-\varepsilon}(\Omega)$-regularity for any small $\varepsilon>0$, where the latter space is the exact Sobolev space for discontinuous functions. The observed half an order convergence of the discrete control confirms our above expectations.    

In Figure \ref{Fig:Distribute_Case5} and Table \ref{Tab:Distribute_Case5} we also present the results for Case 5  posed in a non-convex polygonal domain. We can observe reduced orders of convergence for the $H^1(\Omega)$- and $L^2(\Omega)$-norm errors which are reasonable because the solutions may not be in $H^2(\Omega)$. Moreover, nearly half an order convergence is observed for the control variable.   

\begin{figure}
\includegraphics[scale=0.28]{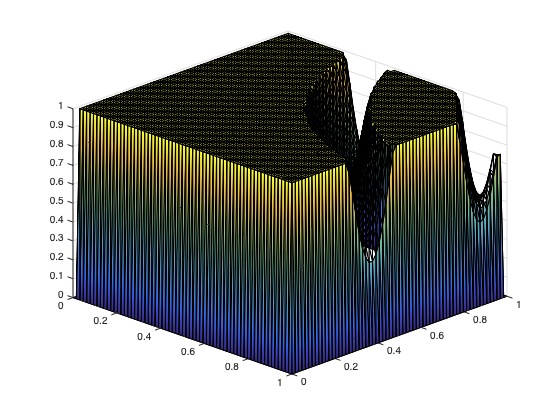}
\includegraphics[scale=0.28]{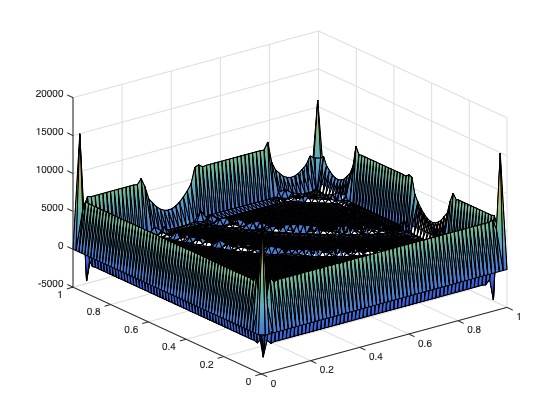}
\caption{The profiles of the discrete state $y_h$ (left) and control $u_h$ (right) for Example \ref{Exm:Distributed}: Case 1. }\label{Fig:Distribute_Case1}
\end{figure}

\begin{figure}
\includegraphics[scale=0.28]{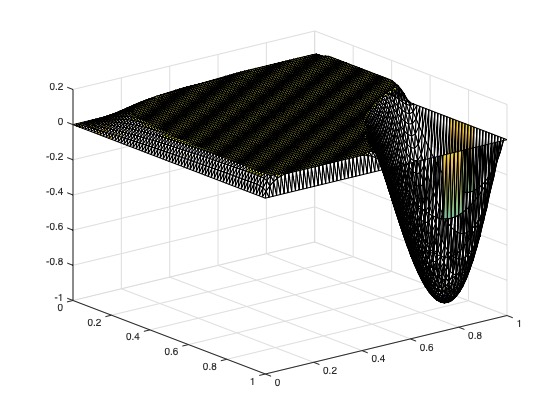}
\includegraphics[scale=0.28]{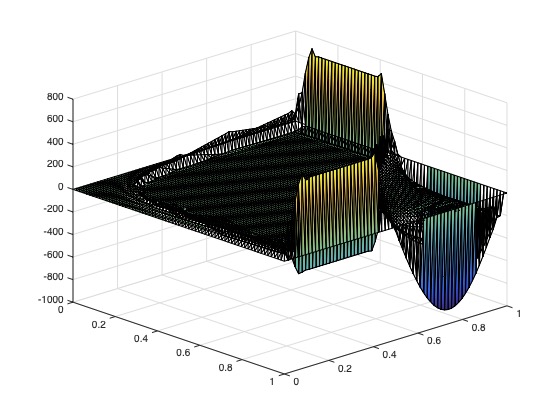}
\caption{The profiles of the discrete state $y_h$ (left) and control $u_h$ (right) for Example \ref{Exm:Distributed}: Case 2. }\label{Fig:Distribute_Case2}
\end{figure}

\begin{figure}
\includegraphics[scale=0.28]{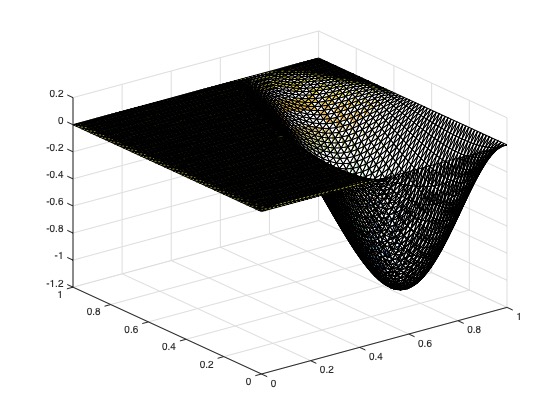}
\includegraphics[scale=0.28]{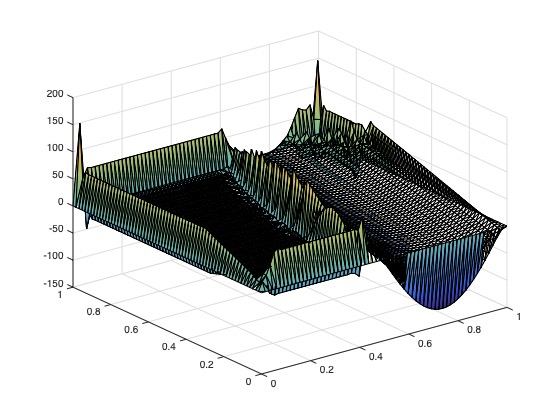}
\caption{The profiles of the discrete state $y_h$ (left) and control $u_h$ (right) for Example \ref{Exm:Distributed}: Case 3. }\label{Fig:Distribute_Case3}
\end{figure}

\begin{figure}
\includegraphics[scale=0.28]{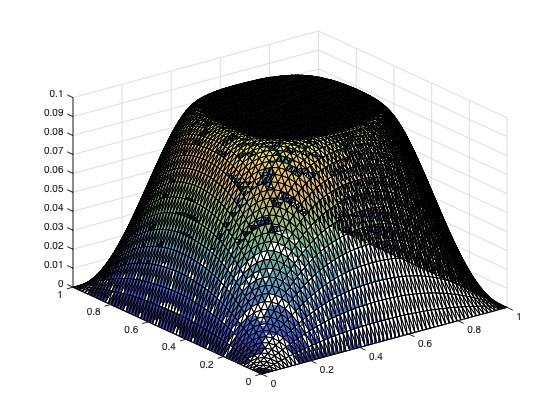}
\includegraphics[scale=0.28]{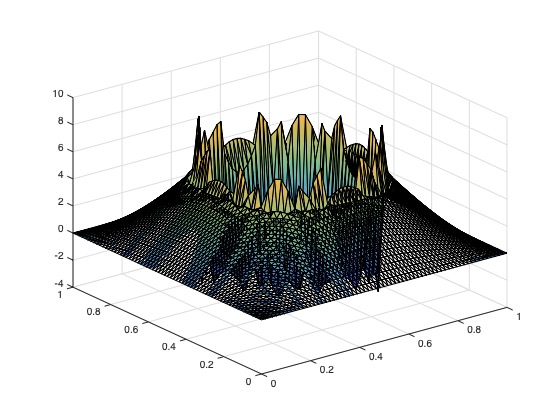}
\caption{The profiles of the discrete state $y_h$ (left) and control $u_h$ (right) for Example \ref{Exm:Distributed}: Case 4. }\label{Fig:Distribute_Case4}
\end{figure}

\begin{figure}
\includegraphics[scale=0.28]{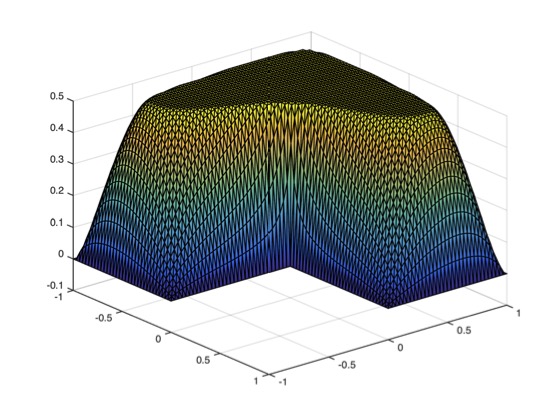}
\includegraphics[scale=0.28]{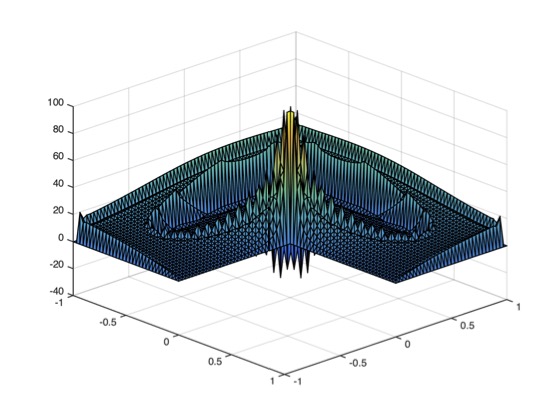}
\caption{The profiles of the discrete state $y_h$ (left) and control $u_h$ (right) for Example \ref{Exm:Distributed}: Case 5. }\label{Fig:Distribute_Case5}
\end{figure}

\begin{table}\scriptsize
\begin{center}
\begin{tabular}{|c|c|c|c|c|c|c|}
\hline
${\rm Dofs}$&$289$&$1089$&$4225$&$16641$&$66049$&$263169$\\
\hline
$\|u-u_h\|_{0,\Omega}$&2.37054e3&2.33463e3&2.06213e3&1.39117e3&9.66214e2&5.66897e2\\
\hline
order&$\setminus$&0.02202 &0.17906 &0.56784 &0.52588&0.76926  \\
\hline
$\|y-y_h\|_{0,\Omega}$&2.22190e-1&1.24713e-1&4.81439e-2&1.06087e-2&2.73754e-3&5.69692e-4\\
\hline
order&$\setminus$&0.83318  &1.37319 &2.18211& 1.95430&2.26463  \\
\hline
$ |y-y_h|_{1,\Omega}$&1.62620e1&1.43833e1&1.00702e1&5.30384e0&2.64436e0&1.19051e0\\
 \hline
order&$\setminus$&0.17711 & 0.51430&0.92498&1.00412&1.15134 \\
\hline
 \end{tabular}
\end{center}
 \caption{Errors of the control $u$ and state $y$ for Example \ref{Exm:Distributed}: Case 1.}\label{Tab:Distribute_Case1}
\end{table}

\begin{table}\scriptsize
\begin{center}
\begin{tabular}{|c|c|c|c|c|c|c|}
\hline
${\rm Dofs}$&$81$&$289$&$1089$&$4225$&$16641$&$66049$\\
\hline
$\|u-u_h\|_{0,\Omega}$&1.21742e2&1.14069e2&1.00602e2&7.50766e1&4.46354e1&2.71775e1\\
\hline
order&$\setminus$&0.09392 &0.18125 &0.42222 &0.75017 &0.71578  \\
\hline
$\|y-y_h\|_{0,\Omega}$&6.98265e-2&2.54640e-2&8.54258e-3&2.41114e-3&5.78624e-4&1.22131e-4\\
\hline
order&$\setminus$&1.45532  &1.57572 &1.82496  & 2.05902&2.24420  \\
\hline
$|y-y_h|_{1,\Omega}$&2.48378e0&1.66600e0&1.02856e0&5.62518e-1&2.77043e-1&1.24375e-1\\
 \hline
order&$\setminus$&0.57615 & 0.69576&0.87065&1.02179&1.15541 \\
\hline
 \end{tabular}
\end{center}
\caption{Errors of the control $u$ and state $y$ for Example \ref{Exm:Distributed}: Case 2.}\label{Tab:Distribute_Case2}
\end{table}

\begin{table}\scriptsize
\begin{center}
\begin{tabular}{|c|c|c|c|c|c|c|}
\hline
${\rm Dofs}$&$81$&$289$&$1089$&$4225$&$16641$&$66049$\\
\hline
$\|u-u_h\|_{0,\Omega}$&2.22769e1&1.95942e1&1.79414e1&1.48502e1&8.27368e0&5.08190e0 \\
\hline
order&$\setminus$&0.18512 &0.12713 &0.27281 &0.84388 &0.70316  \\
\hline
$\|y-y_h\|_{0,\Omega}$&4.51573e-2&1.14652e-2&3.09455e-3&7.96831e-4&1.87014e-4&3.88578e-5  \\
\hline
order&$\setminus$&1.97770 &1.88946&1.95738& 2.09113&2.26687  \\
\hline
$|y-y_h|_{1,\Omega}$&1.02901e0&5.46609e-1&2.92568e-1&1.54741e-1&7.56909e-2&3.41020e-2 \\
 \hline
order&$\setminus$&0.91268 &0.90174&0.91892&1.03166&1.15026 \\
\hline
 \end{tabular}
\end{center}
\caption{Errors of the control $u$ and state $y$ for Example \ref{Exm:Distributed}: Case 3.}\label{Tab:Distribute_Case3}
\end{table}

\begin{table}\scriptsize
\begin{center}
\begin{tabular}{|c|c|c|c|c|c|c|}
\hline
${\rm Dofs}$&$25$&$81$&$289$&$1089$&$4225$&$16641$\\
\hline
$\|u-u_h\|_{0,\Omega}$&1.56696e0&1.00545e0&7.30718e-1&4.71346e-1&3.41755e-1&2.34385e-1\\
\hline
order&$\setminus$&0.64013&0.46045&0.63253&0.46382&0.54408\\
\hline
$\|y-y_h\|_{0,\Omega}$&1.07669e-2&3.31279e-3&7.63086e-4&1.65823e-4&4.32484e-5&1.01038e-5\\
\hline
order&$\setminus$& 1.70048&2.11813& 2.20220&1.93893&2.09775 \\
\hline
$|y-y_h|_{1,\Omega}$&1.76102e-1&9.43473e-2&4.84632e-2&2.41770e-2&1.20412e-2&5.89363e-3\\
 \hline
order&$\setminus$&0.90036 & 0.96109&1.00325&1.00566&1.03075 \\
\hline
 \end{tabular}
\end{center}
\caption{Errors of the control $u$ and state $y$ for Example \ref{Exm:Distributed}: Case 4.}\label{Tab:Distribute_Case4}
\end{table}

\begin{table}\scriptsize
\begin{center}
\begin{tabular}{|c|c|c|c|c|c|c|}
\hline
${\rm Dofs}$&$25$&$81$&$289$&$1089$&$4225$&$16641$\\
\hline
$\|u-u_h\|_{0,\Omega}$&1.94921e1&1.33042e1&9.34247e0&6.31659e0&4.07998e0&2.51172e0
\\
\hline
order&$\setminus$&0.55101 & 0.51001 & 0.56466 & 0.63058 & 0.69989\\
\hline
$\|y-y_h\|_{0,\Omega}$&1.29567e-1&3.43028e-2&1.04353e-2&3.29783e-3&1.05923e-3&2.75509e-4\\
\hline
order&$\setminus$& 1.91730 & 1.71685 & 1.66188 & 1.63850 & 1.94284 \\
\hline
$|y-y_h|_{1,\Omega}$&1.70903e0&9.81369e-1&5.40935e-1&2.96144e-1&1.59543e-1&7.87231e-2\\
 \hline
order&$\setminus$&0.80031 & 0.85934 & 0.86916 & 0.89235 & 1.01909 \\
\hline
 \end{tabular}
\end{center}
\caption{Errors of the control $u$ and state $y$ for Example \ref{Exm:Distributed}: Case 5.}\label{Tab:Distribute_Case5}
\end{table}

\subsection{Dirichlet boundary control problems}

\begin{Example}\label{Exm:Dirichlet}
We consider the elliptic Dirichlet boundary control problem with pointwise state constraints with data given by 
\begin{itemize} 
\item Case 1: $\alpha= 0.01$, $y_d(x_1, x_2) = \sin(\pi x_1)\sin(\pi x_2)$ and $y_b = 0.4$.

\item Case 2: $\alpha = 1e-3$, $y_d(x_1, x_2) = \sin(2\pi x_1x_2)$ and $y_b = 0.1$.

\item Case 3: $\alpha= 0.1$, $y_d(x_1, x_2) = 10(\sin(2x_1)+x_2)$ and $y_b = 0.01$.


\end{itemize}

\end{Example}

For Dirichlet boundary control problems we also list in Figure \ref{Fig:Dirichlet_Case1}-\ref{Fig:Dirichlet_Case3} and Table \ref{Tab:Dirichlet_Case1}-\ref{Tab:Dirichlet_Case3} the profiles and the convergence rates of the state and control approximations, although we did not provide a priori error estimates. Again, we can observe first order convergence for the energy norm error and second order convergence for the $L^2(\Omega)$-norm error. The observed convergence rates are consistent with our expectations but a theoretical justification is still missing. Since the control is the Dirichlet trace of the state, a standard trace theorem implies the $O(h^{3/2})$-order convergence for the control in $L^2(\Gamma)$-norm if we have $H^2$-regularity for the state. However, we observe a better second order convergence rate that is the best possible rate for linear finite element approximations. Moreover, compared to the distributed and Neumann boundary control problems, the Dirichlet boundary control is continuous along the boundary. In fact, we have
$u\in H^{1/2}(\Gamma)\cap \prod\limits_{i=1}^4H^{3/2}(\Gamma_i)$ if $y\in H^2(\Omega)$. When the state $y$ has higher regularity, say $H^{5/2-\varepsilon}(\Omega)$, we have $u\in \prod\limits_{i=1}^4H^{2-\varepsilon}(\Gamma_i)$ which may partially explain the observed second order convergence. The above analysis suggests that our proposed method is very suitable for solving Dirichlet boundary control problems with pointwise state constraints.  

\begin{figure}
\includegraphics[scale=0.28]{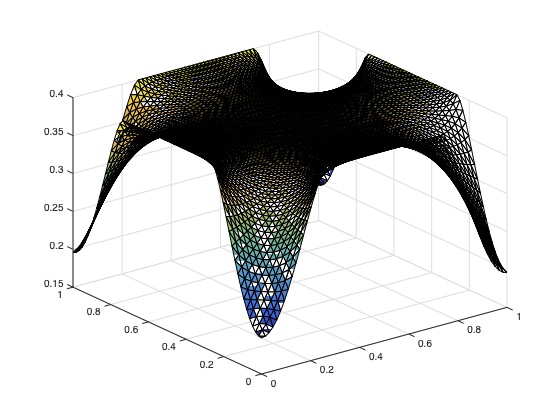}
\includegraphics[scale=0.28]{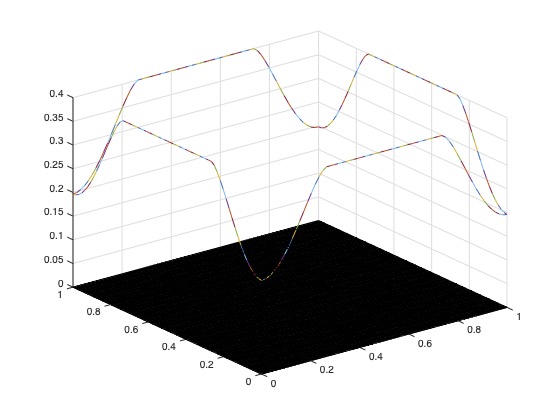}
\caption{The profiles of the discrete state $y_h$ (left) and control $u_h$ (right) for Example \ref{Exm:Dirichlet}: Case 1. }\label{Fig:Dirichlet_Case1}
\end{figure}

\begin{figure}
\includegraphics[scale=0.28]{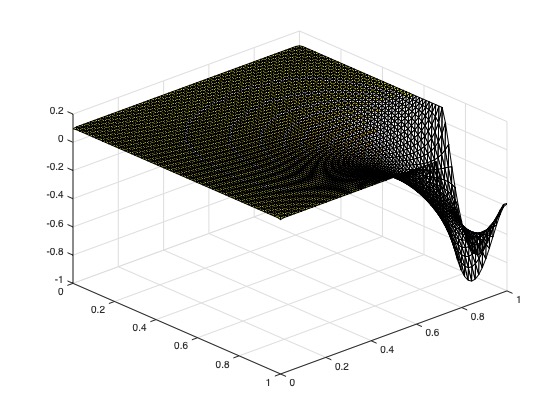}
\includegraphics[scale=0.28]{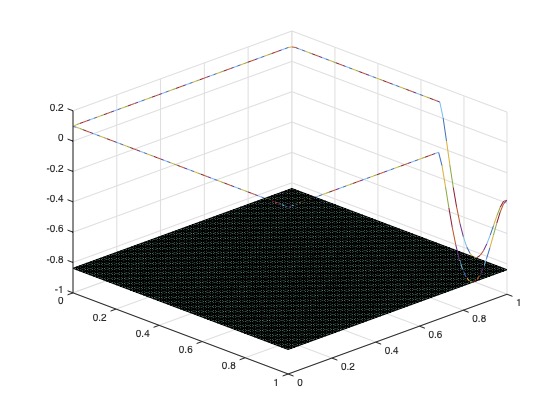}
\caption{The profiles of the discrete state $y_h$ (left) and control $u_h$ (right) for Example \ref{Exm:Dirichlet}: Case 2. }\label{Fig:Dirichlet_Case2}
\end{figure}
\begin{figure}
\includegraphics[scale=0.28]{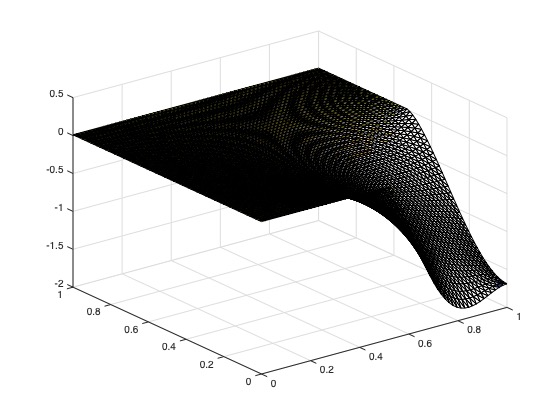}
\includegraphics[scale=0.28]{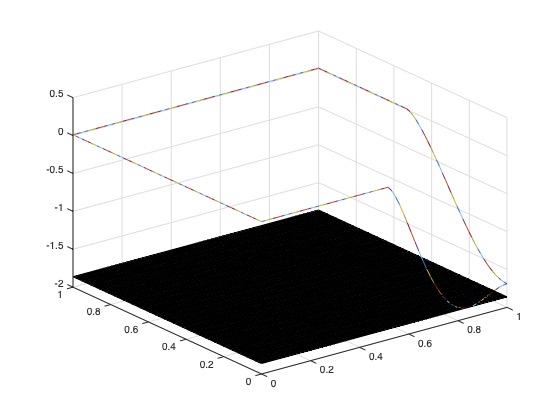}
\caption{The profiles of the discrete state $y_h$ (left) and control $u_h$ (right) for Example \ref{Exm:Dirichlet}: Case 3. }\label{Fig:Dirichlet_Case3}
\end{figure}

\begin{table}\scriptsize
\begin{center}
\begin{tabular}{|c|c|c|c|c|c|c|}
\hline
${\rm Dofs}$&$81$&$289$&$1089$&$4225$&$16641$&$66049$\\
\hline
$\|u-u_h\|_{0,\Omega}$&1.65574e-2 & 5.00329e-3 & 1.93327e-3 & 5.94725e-4 & 2.32788e-4&1.25012e-4\\
\hline
order&$\setminus$&1.72653 & 1.37183 & 1.70075 & 1.35321 & 0.89695  \\
\hline
$\|y-y_h\|_{0,\Omega}$&5.37217e-3 & 1.32124e-3 & 3.17696e-4 & 6.85380e-5 & 4.15074e-5 & 3.52024e-5\\
\hline
order&$\setminus$&2.02361 & 2.05617 & 2.21267& 0.72354 & 0.23769  \\
\hline
$|y-y_h|_{1,\Omega}$&1.98603e-1 & 1.04774e-1 & 5.43983e-2 & 2.78890e-2 & 1.37734e-2 & 6.19718e-3\\
 \hline
order&$\setminus$&0.92260 & 0.94565 & 0.96387 & 1.01781& 1.15220\\
\hline
 \end{tabular}
\end{center}
 \caption{Errors of the control $u$ and state $y$ for Example \ref{Exm:Dirichlet}: Case 1.}\label{Tab:Dirichlet_Case1}
\end{table}

\begin{table}\scriptsize
\begin{center}
\begin{tabular}{|c|c|c|c|c|c|c|}
\hline
${\rm Dofs}$&$81$&$289$&$1089$&$4225$&$16641$&$66049$\\
\hline
$\|u-u_h\|_{0,\Omega}$&1.21661e-1 & 2.56414e-2 & 9.83374e-3 & 3.30843e-3 & 9.95541e-4 & 2.49940e-4\\
\hline
order&$\setminus$&2.24632 & 1.38266 & 1.57159 & 1.73260 & 1.99390 \\
\hline
$\|y-y_h\|_{0,\Omega}$&1.76084e-2 & 5.26225e-3 & 1.29948e-3 & 3.28244e-4 & 8.29667e-5 & 2.06652e-5\\
\hline
order&$\setminus$&1.74251 & 2.01775 & 1.98509 & 1.98416 & 2.00533\\
\hline
$|y-y_h|_{1,\Omega}$&7.84493e-1 & 4.92971e-1 & 2.67145e-1 & 1.39058e-1 & 6.94337e-2 & 3.14369e-2\\
 \hline
order&$\setminus$&0.67026 & 0.88388 & 0.94194 & 1.00197 & 1.14318 \\
\hline
 \end{tabular}
\end{center}
\caption{Errors of the control $u$ and state $y$ for Example \ref{Exm:Dirichlet}: Case 2.}\label{Tab:Dirichlet_Case2}
\end{table}

\begin{table}\scriptsize
\begin{center}
\begin{tabular}{|c|c|c|c|c|c|c|}
\hline
${\rm Dofs}$&$81$&$289$&$1089$&$4225$&$16641$&$66049$\\
\hline
$\|u-u_h\|_{0,\Omega}$&5.51180e-2 & 2.04315e-2 & 5.82645e-3 & 9.64193e-4 & 2.61135e-4 & 1.09503e-4 \\
\hline
order&$\setminus$& 1.43173 & 1.81011 & 2.59522 & 1.88453 & 1.25383 \\
\hline
$\|y-y_h\|_{0,\Omega}$&2.12443e-2 & 6.25472e-3 & 1.52660e-3 & 2.87902e-4 & 8.39574e-5 & 3.70974e-5 \\
\hline
order&$\setminus$&1.76406 & 2.03462 & 2.40667 & 1.77785 & 1.17834 \\
\hline
$|y-y_h|_{1,\Omega}$&6.48880e-1 & 3.47359e-1 & 1.79600e-1 & 9.13242e-2 & 4.48402e-2 & 2.01002e-2 \\
 \hline
order&$\setminus$&0.90152 & 0.95165 & 0.97571 & 1.02620 & 1.15758\\
\hline
 \end{tabular}
\end{center}
\caption{Errors of the control $u$ and state $y$ for Example \ref{Exm:Dirichlet}: Case 3.}\label{Tab:Dirichlet_Case3}
\end{table}

\subsection{Neumann boundary control problems}
\begin{Example}\label{Exm:Neumann}
We consider the elliptic Neumann boundary control problem with pointwise state constraints with data given by 
\begin{itemize} 
\item Case 1: $\alpha= 0.01$, $y_d(x_1, x_2) = \sin(\pi x_1)\sin(\pi x_2)$ and $y_b = 0.4$.
\end{itemize}
Other numerical results for Neumann boundary controls with pointwise state constraints can be found in, e.g., \cite{KrumbiegelMeyerRosch}.
\end{Example}

For Neumann boundary control problems, we present in  Figure \ref{Fig:Neumann_Case1} the profiles of the discrete states and controls, and in Table \ref{Tab:Neumann_Case1} the convergence rates of the state and control approximations, where the discrete controls are computed by (\ref{Neumann_uh}). Again, for the state variable we can observe a first order convergence for the energy norm error and a second order convergence for the $L^2(\Omega)$-norm error, all of which are in agreement with our expectations. However, we can only observe reduced first order convergence rate for the control in $L^2(\Gamma)$-norm. If the state $y\in H^2(\Omega)$ when $\Omega$ is convex and $y_d\in L^2(\Omega)$,  from the standard trace theorem we have $u\in\prod\limits_{i=1}^4H^{1/2}(\Gamma_i)$. Moreover, since $y_d$ is smooth and $y_b$ is a constant, we may expect the $W^{s,p}(\Omega)$ regularity for the state $y$ for any $s<2+1/p$, $p\in [1,\infty)$. So that in dimension two we have $y\in H^{5/2-\varepsilon}(\Omega)$ and  $u\in\prod\limits_{i=1}^4H^{1-\varepsilon}(\Gamma_i)$ for any $\varepsilon>0$. Therefore, the control is not globally continuous on polygonal domains but exhibits local higher regularity. However, there are only finite many discontinuous points (in fact, 4 points) on the boundary. This may partially explain the observed first order convergence of the discrete control.

\begin{figure}
\includegraphics[scale=0.28]{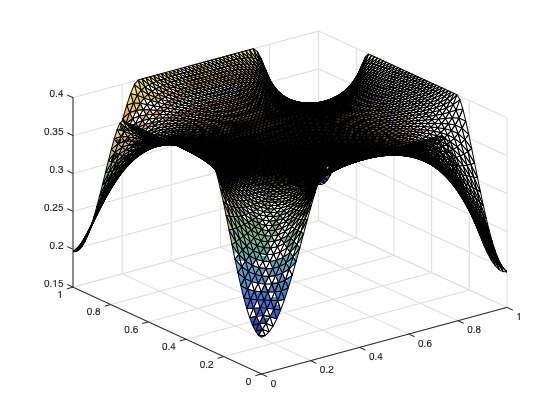}
\includegraphics[scale=0.28]{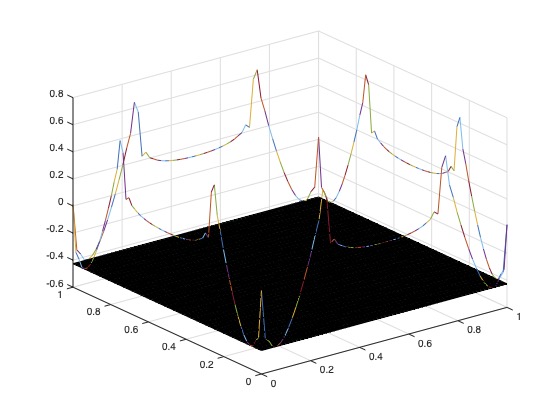}
\caption{The profiles of the discrete state $y_h$ (left) and control $u_h$ (right) for Example \ref{Exm:Neumann}: Case 1. }\label{Fig:Neumann_Case1}
\end{figure}

\begin{table}\scriptsize
\begin{center}
\begin{tabular}{|c|c|c|c|c|c|c|}
\hline
${\rm Dofs}$&$41$&$145$&$545$&$2113$&$8321$&$33025$\\
\hline
$\|u-u_h\|_{0,\Omega}$&5.84845e-1 & 3.38911e-1 & 1.93965e-1 & 9.14250e-2 & 4.26479e-2 & 1.80134e-2\\
\hline
order&$\setminus$&0.78715&0.80511&1.08514&1.10011&1.24341\\
\hline
$\|y-y_h\|_{0,\Omega}$&9.71945e-3 & 2.62477e-3 & 9.07135e-4 & 1.92152e-4 & 4.71050e-5 & 7.92093e-6\\
\hline
order&$\setminus$&1.88868&1.53280&2.23907&2.02830&2.57214 \\
\hline
$\|y-y_h\|_{1,\Omega}$&2.06528e-1 & 1.27615e-1 & 7.25363e-2 &  3.73501e-2  & 1.85047e-2 & 8.32360e-3\\
 \hline
order&$\setminus$&0.69454&0.81502&0.95759&1.01322&1.15261\\
\hline
 \end{tabular}
\end{center}
\caption{Errors of the control $u$ and state $y$ for Example \ref{Exm:Neumann}: Case 1.}\label{Tab:Neumann_Case1}
\end{table}

\subsection{Elliptic distributed optimal control problems with gradient state constraints}
\begin{Example}\label{Exm:gradient}
We set $\Omega=(0,1)^2$, $\alpha=0.1$, $y_b=1$ and $y_d=2\sin(\pi x_1)\sin(\pi x_2)$.
\end{Example}

To solve the elliptic variational inequality with gradient constraints we use the semi-smooth Newton method proposed in \cite{Anyyeva}, which is achieved by penalizing the gradient constraints of the minimization problem in the objective functional and solving the latter one by a semi-smooth Newton method. To be more specific, we solve the unconstrained optimization problem
\begin{eqnarray}
\min\limits_{y\in H_0^1(\Omega)}J_\gamma(y)&={1\over 2}\int_\Omega (\alpha|\nabla y|^2+y^2)dx-\int_{\Omega}y_dydx +\frac{\gamma}{2}\int_\Omega \max(0,|\nabla y|^2-y_b^2)^2dx.\label{Min_pen}
\end{eqnarray}
The functional $J_\gamma$ is first order Fr\'{e}chet differentiable. We use the so-called Newton derivative to apply Newton's method to the Euler-Lagrange system of the unconstrained minimization problem. This strategy is actually the active set strategy as pointed out in \cite{Kunisch}. 

At each Newton step we solve the following equation for $y_\gamma^{k+1}\in H_0^1(\Omega)$:
\begin{eqnarray}
&&\int_\Omega(\alpha+2\gamma\chi_{\mathcal{A}}^k(|\nabla y_\gamma^k|^2-y_b^2))\nabla y_\gamma^{k+1}\cdot\nabla v+4\gamma\int_\Omega\chi_{\mathcal{A}}^k(\nabla y_\gamma^{k+1})^T\nabla y_\gamma^k \otimes\nabla y_\gamma^k\nabla v\nonumber\\
&&+\int_\Omega y_\gamma^{k+1} v=4\gamma\int_\Omega\chi_{\mathcal{A}}^k|\nabla y_\gamma^k|^2\nabla y_\gamma^k\cdot\nabla v+\int_\Omega y_dv,\ \ \ \forall v\in H_0^1(\Omega),\label{gradient_constrained_penalty}
\end{eqnarray}
where $\chi_{\mathcal{A}}^k$ is the characteristic function of the active set $\mathcal{A}$. We also use the method of continuation to increase $\gamma$ gradually from $1$ up to $1.0e+5$, so that the solutions associated to smaller $\gamma$ provide good initial guesses for the solutions associated to larger $\gamma$.

\begin{figure}
\includegraphics[scale=0.28]{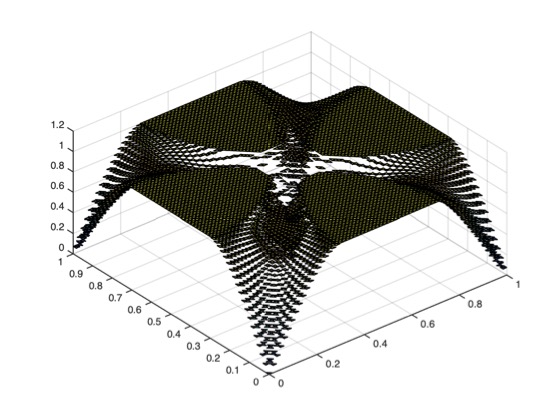}
\includegraphics[scale=0.28]{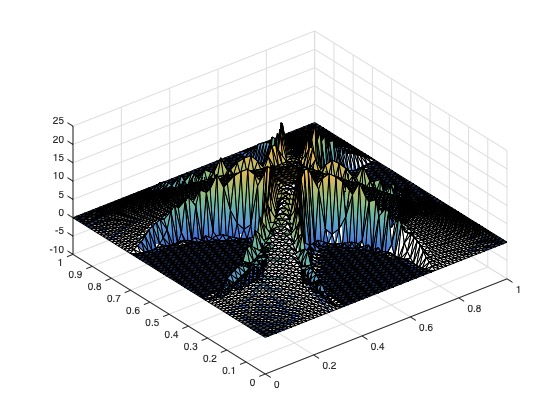}
\caption{The profiles of the  discrete gradient state $|\nabla y_h|$ (left) and control $u_h$ (right) for Example \ref{Exm:gradient}. }\label{Fig:gradient_1}
\end{figure}

\begin{figure}
\includegraphics[scale=0.28]{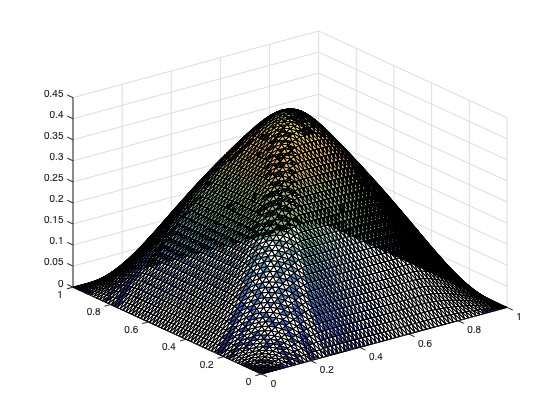}
\caption{The profiles of the  discrete state $ y_h$ for Example \ref{Exm:gradient}. }\label{Fig:gradient_2}
\end{figure}

In Figure \ref{Fig:gradient_1} and \ref{Fig:gradient_2} we present the profiles of the discrete gradient state, state and control. We can observe that  the control oscillates on the boundaries of the active sets of the gradient state constraints. In Table \ref{Table:gradient} we show the convergence rates for the state and control. We can observe first and second order convergence rates for respectively the $H^1$- and $L^2$-norm errors of the state variable, which are optimal with respect to the finite element space. For the error estimate of second order elliptic variational inequality with gradient constraints, we refer to \cite{FalkMercier} for a dual formulation where first order convergence rate was proved for the $L^2$-norm of the gradient approximation. As mentioned in \cite[p. 142]{FalkMercier} a suboptimal  error estimate was obtained in \cite{Falk1971} for a standard approximation of elliptic variational inequality with gradient constraints by using piecewise linear finite elements. On the other hand, we can observe nearly half an order convergence for the control. This is consistent with the pointwise state constraints case, and shows a $L^2(\Omega)$ control and a $H^2(\Omega)$ state.

\begin{table}
\scriptsize
\begin{center}
\begin{tabular}{|c|c|c|c|c|c|c|}
\hline
${\rm Dofs}$&$81$&$289$&$1089$&$4225$&$16641$&$66049$\\
\hline
$\|u-u_h\|_{0,\Omega}$&2.61112e0 & 1.64978e0 & 1.29265e0 & 9.16654e-1 & 6.96937e-1 & 4.76736e-1\\
\hline
order&$\setminus$&0.66240 & 0.35194 & 0.49588 & 0.39535 &  0.54784\\
\hline
$\|y-y_h\|_{0,\Omega}$& 1.09187e-2 & 2.30309e-3 & 4.41287e-4 & 1.19263e-4 & 2.69430e-5 & 5.70502e-6\\
\hline
order&$\setminus$& 2.24516 & 2.38378 & 1.88757 & 2.14616 & 2.23961 \\
\hline
$|y-y_h|_{1,\Omega}$& 2.23013e-1 & 1.15837e-1 & 5.92635e-2 & 2.97473e-2 & 1.45820e-2 & 6.53606e-3\\
 \hline
order&$\setminus$&0.94503 & 0.96688 & 0.99439 & 1.02857 & 1.15770\\
\hline
 \end{tabular}
\end{center}
\caption{Errors of the control $u$ and state $y$ for Example \ref{Exm:gradient}.}\label{Table:gradient}
\end{table}

\section{Concluding Remarks}
In this paper we proposed a new finite element method for solving elliptic optimal control problems with pointwise state constraints based on energy space regularizations. By using the equivalent representations of the energy space norms, we finally obtained second order elliptic variational inequalities, which can be solved efficiently by using standard finite element methods. The validation of the method has been verified by extensive numerical experiments. We also derived a priori error estimates for the distributed control problems with pointwise state constraints.


 \medskip
\end{document}